\documentclass{article}
\usepackage{amsmath,amssymb,latexsym,amsthm,mathrsfs}

\newtheorem{theo}{Theorem}[section]
\newtheorem{pro}[theo]{Proposition}
\newtheorem{lem}[theo]{Lemma}
\newtheorem{cor}[theo]{Corollary}

\newcommand{\ra}{\rightarrow}
\theoremstyle{definition}
\newtheorem{defin}[theo]{Definition}
\newtheorem{exa}{Example}[section]

\theoremstyle{remark}
\newtheorem{rem}[theo]{Remark}

\begin{document}

\title{Random walks on free solvable groups}
\author{Laurent Saloff-Coste\thanks{%
Both authors partially supported by NSF grant DMS 1004771} \\
{\small Department of mathematics}\\
{\small Cornell University} \and Tianyi Zheng \\
{\small Department of Mathematics}\\
{\small Cornell University} }
\maketitle

\begin{abstract}
For any finitely generated group $G$, let $n\mapsto \Phi_G(n)$ 
be the function that describes the rough asymptotic behavior of the 
probability of return to the identity element at 
time $2n$ of a symmetric simple random walk on $G$ (this  
is an invariant of quasi-isometry). We determine this function
when $G$ is the free solvable group $\mathbf{S}_{d,r}$
of derived length $d$ on $r$ generators and some other related groups. 
\end{abstract}

\noindent{\bf Math Subject Classification:  20F69,  60J10} 

\section{Introduction} \setcounter{equation}{0}
\subsection{The random walk group invariant $\Phi_G$}
Let $G$ be a finitely generated group. Given a probability measure $\mu$ on $G$, 
the random walk driven by $\mu$ (started at the identity element $e$ of $G$)
is the $G$-valued random process $X_n=\xi_1\dots \xi_n$ where 
$(\xi_i)_1^\infty$ is a  sequence of independent identically distributed 
$G$-valued random variables with law $\mu$. If $u*v(g)=\sum_h u(h)v(h^{-1}g)$ 
denotes the convolution of two functions $u$ and $v$ on $G$ then the probability
that $X_n=g$ is given by $\mathbf P^\mu_e(X_n=g)=\mu^{(n)}(g)$ where $\mu^{(n)}$ 
is the $n$-fold convolution of $\mu$. 

Given a symmetric set of generators $S$, the word-length $|g|$ of $g\in G$
is the minimal length of a word representing $g$ in the elements of $S$. 
The associated volume growth function, $r\mapsto V_{G,S}(r)$, 
counts the number of elements of $G$ with $|g|\le r$. 
The word-length induces a left invariant metric on $G$ which is also the graph 
metric on the Cayley graph $(G,S)$. 
A quasi-isometry between two Cayley graghs $(G_i,S_i)$, $i=1,2$, say, from $G_1$ to $G_2$, 
is a map $q: G_1\ra G_2$ such that 
$$C^{-1}d_2(q(x),q(y))\le d_1(x,y)\le C(1+d_2(q(x),q(y)))$$  and 
$\sup_{g,\in G_2}\{d_2(g,q(G_1))\le C$ for some finite positive constant $C$. This induces an equivalence 
relation on Cayley graphs. In particular, 
$(G,S_1),(G,S_2)$ are quasi-isometric for 
any choice of generating sets $S_1,S_2$. See, e.g., \cite{dlH} for more details. 

Given two monotone functions $\phi,\psi$, write $\phi\simeq \psi$ 
if there are constants
$c_i\in (0,\infty)$, $1\le i\le 4$, 
such that $c_1\psi(c_2t)\le  \phi(t)\le c_3 \psi(c_4t)$ 
(using integer values if $\phi,\psi$ are defined on $\mathbb N$). 
If $S_1,S_2$ are two symmetric generating sets for $G$, 
then $V_{G,S_1}\simeq V_{G,S_2}$. We use the notation $V_G$ to denote either 
the $\simeq$-equivalence class of $V_{G,S}$ or any one of its representatives. 
The volume growth function $V_G$ is one of the simplest quasi-isometry invariant 
of a group $G$.

By \cite[Theorem 1.4]{Pittet2000}, if $\mu_i$, $i=1,2$, are 
symmetric (i.e., $\mu_i(g)=\mu_i(g^{-1})$ for all $g\in G$) finitely supported
probability measures with generating support, 
then the functions $n\mapsto \phi_i(n)=\mu_i^{(2n)}(e)$ satisfy 
$\phi_1\simeq \phi_2$. By definition, we denote by $\Phi_G$ any function
that belongs to the $\simeq$-equivalence class of $\phi_1\simeq\phi_2$.  
In fact, $\Phi_G$ is an invariant of quasi-isometry.
Further, if $\mu$ is a symmetric probability measure with generating support and finite second 
moment $\sum_G|g|^2\mu(g)<\infty$ then $\mu^{(2n)}(e)\simeq \Phi_G(n)$.
See \cite{Pittet2000}.

\subsection{Free solvable groups}
This work is concerned with finitely generated solvable groups. 
Recall that $G^{(i)}$, the derived series of $G$, is defined inductively by $G^{(0)}=G$, 
$G^{(i)}=[G^{(i-1)},G^{(i-1)}]$. A group is solvable if $G^{(i)}=\{e\}$ 
for some $i$ and the smallest such $i$ is the derived length of $G$. 
A group $G$ 
is polycyclic if it admits a normal descending  series 
$G=N_0\supset N_1\supset\dots\supset N_k=\{e\}$ such that each of 
the quotient $N_i/N_{i+1}$ is cyclic.  The 
lower central series $\gamma_j(G)$, $j\ge 1$, of a group $G$ is obtained by setting
$\gamma_1(G)=G$ and $\gamma_{j+1}=[G,\gamma_j(G)]$. A group $G$ is nilpotent 
of nilpotent class $c$ if $\gamma_c(G)\neq\{e\}$ and $\gamma_{c+1}(G)=\{e\}$.
Finitely generated nilpotent groups are 
polycyclic and polycyclic groups are solvable. 

Recall the following well-known facts.
If $G$ is a finitely generated solvable group then either 
$G$ has polynomial volume growth $V_G(n)\simeq n^D$ for some $D=0,1,2,\dots$, or
$G$ has exponential volume growth $V_G(n)\simeq \exp(n)$. 
See, e.g., \cite{dlH} and the references therein. 
If $V_G(n)\simeq n^D$ then $G$ is virtually nilpotent and $\Phi_G(n)\simeq n^{-D/2}$.
If $G$ is polycyclic with 
exponential volume growth then $\Phi_G(n)\simeq \exp(-n^{1/3})$. 
See \cite{Alexpol,Hebish1993,Varnil,VarSendai,VSCC} and the references given there.
However, among solvable groups
of exponential volume growth, many other behaviors than those described 
above are known to occur. 
See, e.g., 
\cite{Erschler2006,Pittet2002,SCnotices}. 
Our main result is the following theorem. Set 
$$\log_{[1]}n= \log (1+n) \mbox{ and } \log_{[i]}(n)=\log (1+ \log_{[i-1]} n).$$

\begin{theo} \label{th-freesol}
Let $\mathbf{S}_{d,r}$ be the free solvable group of derived length $d$ on $r$ generators, 
that is, $\mathbf{S}_{d,r}=\mathbf{F}_r/\mathbf{F}_r^{(d)}$ where 
$\mathbf{F}_r$ is the free group on $r$ generators, $r\ge 2$.
\begin{itemize}
\item If $d=2$ (the free metabelian case) then
$$\Phi_{\mathbf{S}_{2,r}}(n)\simeq \exp \left(- n^{r/(r+2)}(\log n)^{2/(r+2)}\right).$$
\item If $d>2$  then
$$\Phi_{\mathbf{S}_{d,r}}(n)\simeq \exp \left(- n \left(
\frac{\log_{[d-1]} n}{\log_{[d-2]} n}\right)^{2/r}\right).$$
\end{itemize}
\end{theo}
In the case $d=2$, this result is known and due 
to Anna Erschler who computed the 
F\o lner function of $\mathbf{S}_{2,r}$ in an unpublished work based on the ideas
developed in \cite{Erschler2006}. 
We give a different proof.   
As far as we know, the F\o lner function of $\mathbf{S}_{d,r}$, $d>2$ 
is not known.

Note that if $G$ is $r$-generated and solvable of length at most $d$ then
there exists $c,k\in(0,\infty)$ such that
$\Phi_G(n)\ge c\Phi_{\mathbf{S}_{d,r}}(kn)$.

\subsection{On the groups of the form $\mathbf{F}_r/[N,N]$}
The first statement in Theorem \ref{th-freesol} 
can be generalized as follows. Let 
$N$ be a normal subgroup of $\mathbf{F}_r$ and consider the tower 
of $r$ generated groups $\Gamma_d(N)$ defined by 
$\Gamma_d(N)=\mathbf{F}_r/N^{(d)}$.
Given information about $\Gamma_1(N)=\mathbf{F}_r/N$, more precisely, about the pair $(\mathbf F_r,N)$,
one may hope to determine 
$\Phi_{\Gamma_d(N)}$ (in Theorem \ref{th-freesol}, $N=[\mathbf F_r,\mathbf F_r]$ 
and  $\Gamma_1(N)=\mathbb Z^r$). 
The following theorem captures some of the results we obtain in 
this direction when $d=2$.
Further examples are given in Section \ref{sec-misc}.
\begin{theo} \label{th-NNN}
Let $N \unlhd \mathbf{F}_r$, $\Gamma_1(N)=\mathbf{F}_r/N$  
and $\Gamma_2(N)=\mathbf{F}_r/[N,N]$ 
as above.  
\begin{itemize}
\item  Assume that $\Gamma_1(N)$ is nilpotent of
volume growth of degree $D\ge 2$. Then we have
$$\Phi_{\Gamma_2(N)}(n)\simeq \exp\left(-n^{D/(D+2)}(\log n)^{2/(D+2)}\right).$$
\item  Assume that 
\begin{itemize}
\item either $\Gamma_1(N)= \mathbb Z_q\wr \mathbb Z$
with presentation $\langle a,t| a^q, [a,t^{-n}at^n], \,n\in 
\mathbb Z\rangle$,
\item or $\Gamma_1(N)=\mbox{\em BS}(1,q)$ with presentation $
\langle a,b| a^{-1}ba=b^q\rangle$.  
  \end{itemize}
Then we have
$$\Phi_{\Gamma_2(N)}(n)\simeq \exp\left(-\frac{n}{(\log n)^2} \right).$$
\end{itemize}
\end{theo}
In Section 5, Theorem \ref{th-polyc}, the result stated here for 
$\mbox{BS}(1,q)$ is extended to any polycyclic group of exponential 
volume growth equipped with a standard polycyclic presentation.

Obtaining results for $d\ge 3$ is not easy. The only example we treat 
beyond the case  $N=[\mathbf F_r,\mathbf F_r]$ 
contained in Theorem \ref{th-freesol}, i.e., $\Gamma_d(N)=\mathbf S_{d,r}$, 
is the case when $N=\gamma_c(\mathbf F_r)$. 
See Theorem \ref{th-fs+}.

\begin{rem} Fix the presentation $\mathbf F_r/N=\Gamma_1(N)$.
Let $\boldsymbol \mu$ be the probability measure driving the
lazy simple random walk $(\xi_n)_0^\infty$ on $\mathbf F_r$ so that
$$\mathbf P^{\boldsymbol \mu}_\mathbf e(\xi_n=
\mathbf g)=\boldsymbol \mu^{(n)}(\mathbf g).$$ 
Let $X=(X_n)_0^\infty$ and 
$Y=(Y_n)_0^\infty$ be the projections on $\Gamma_2(N)$ 
and $\Gamma_1(N)$, respectively so that
$$\Phi_{\Gamma_2(N)}(n)\simeq 
\mathbf P^{\boldsymbol \mu}_\mathbf e(X_n=e) \mbox{ and }
\Phi_{\Gamma_1(N)}(n)\simeq 
\mathbf P^{\boldsymbol \mu}_\mathbf e(Y_n= \bar{e})$$
where $e$ (resp. $\bar{e}$) is the identity element in 
$\Gamma_2(N)$ (resp. $\Gamma_1(N)$.) 
By the flow interpretation of the group $\Gamma_2(N)$ developed in 
\cite{DLS,Myasnikov2010} and reviewed in Section \ref{sec-flow} below,
$$\mathbf P^{\boldsymbol \mu}_\mathbf e (X_n=e)
=\mathbf P^{\boldsymbol \mu}_\mathbf e(Y\in \mathfrak B_n)$$
where $\mathfrak B_n$ is the event that, at time $n$, 
every oriented edge of the marked Cayley graph $\Gamma_1(N)$ has been traversed
an equal number of times in both  directions.   For instance,
if $\Gamma_1(N)=\mathbb Z^r$, the estimate $\Phi_{\Gamma_2(N)}(n)\simeq \exp(-n^{r/(2+r)}(\log n)^{2/(2+r)})$ also gives the order of magnitude of 
the probability that a simple random walk on $\mathbb Z^r$ returns to its starting point at time $n$ having crossed each edge an equal number of time in both direction.
\end{rem}

\subsection{Other random walk invariants} \label{sec-otherinv}
Let $|g|$ be the word-length of $G$ with respect to some fixed finite
symmetric generating set and $\rho_\alpha(g)=(1+|g|)^\alpha$. 
In \cite{Bendikov}, for any finitely generated group $G$ and real 
$\alpha\in (0,2)$, the non-increasing function 
\begin{equation*}
\widetilde{\Phi}_{G,\rho_\alpha}: \mathbb{N}\ni n\rightarrow \widetilde{\Phi}%
_{G,\rho_\alpha}(n)\in (0,\infty)
\end{equation*}
is defined  in such a way that it provides the best possible lower bound 
\begin{equation*}
\exists\,c,k\in (0,\infty),\;\forall\, n,\;\;\mu^{(2n)}(e)\ge c 
\widetilde{\Phi}_{G,\rho_\alpha}(kn),
\end{equation*}
valid for every symmetric probability measure $\mu$ on $G$ 
satisfying the weak-$\rho_\alpha$-moment
condition 
\begin{equation*}
W(\rho_\alpha,\mu)=\sup_{s>0}\{ s\mu(\{g: \rho_\alpha(g)>s\})\}<\infty.
\end{equation*}
It is well known and easy to see (using Fourier transform techniques) that 
\begin{equation*}
\widetilde{\Phi}_{\mathbb{Z}^r,\rho_\alpha}(n)\simeq n^{-r/\alpha}.
\end{equation*}
It is proved in \cite{Bendikov} that 
$\widetilde{\Phi}_{G,\rho_\alpha}(n)\simeq n^{-D/\alpha}$ if $G$ has polynomial volume growth of degree $D$ and that
$\widetilde{\Phi}_{G,\rho_\alpha}(n)\simeq \exp\left( -n^{-1/(1+\alpha)}\right)$ 
if $G$ is polycyclic of exponential volume growth. 
We prove the following result. 
\begin{theo} \label{th-freesola}
For any $\alpha\in (0,2)$,
$$\widetilde{\Phi}_{\mathbf{S}_{2,r},\rho_\alpha}(n)\simeq \exp\left(-n^{r/(r+\alpha)}(\log n)^{\alpha/(r+\alpha)}\right).$$
\end{theo}
The lower bound in this theorem  follows from 
Theorem \ref{th-freesol} and \cite{Bendikov}. Indeed, for $d>2$,
Theorem \ref{th-freesol} and \cite[Theorem 3.3]{Bendikov} also give
$$\widetilde{\Phi}_{\mathbf{S}_{d,r},\rho_\alpha}(n)
\geq c \exp\left(- Cn
\left(\frac{\log_{[d-1]}n}{\log_{[d-2]}n}\right)^{\alpha/r}\right).
$$
The upper bound in Theorem \ref{th-freesola} is obtained by 
studying random walks driven by measures that are not finitely supported.
The fact that the techniques we develop below can be applied successfully in certain
cases of this type is worth noting. 
Proving an upper bound matching the lower bound given 
above for $\widetilde{\Phi}_{\mathbf{S}_{d,r},\rho_\alpha}$ with $d>2$ is an 
open problem. 

\subsection{Wreath products and Magnus embedding} \label{sub-wr}
Let $H,K$ be countable groups. Recall that the wreath product $K\wr H$ (with base $H$) 
is the semidirect product of the 
algebraic direct sum $K_H=\sum_{h\in H}K_h$ of $H$-indexed copies of $K$ 
by $H$ where $H$ acts on $K_H$ by translation of the indices. More precisely,
elements of $K\wr H$ are pair $(f,h)\in K_H\times H$ and
$$(f,h)(f',h')=(f\tau_h f',hh')$$ where $\tau_hf_x=f_{h^{-1}x}$ if 
$f=(f_x)_{x\in H}\in K_H$ (recall that, by definition,  only finitely many $f_x$ 
are not the identity element $e_K$  in $K$).  In the context of random walk 
theory, the group $H$ is called the base-group and the group $K$ the lamp-group 
of $K\wr H$ (an element $(f,h) \in K\wr H$ can understood as a finite 
lamp configuration $f$ over $H$ together with the position $h$ of the 
``lamplighter'' on the base $H$). Given  probability measures $\eta$ on 
$K$ and $\mu$ on $H$, the switch-walk-switch random walk on $K\wr H$ is driven by the measure $\eta*\mu*\eta$ and has the following interpretation. At each step,
the lamplighter switches the lamp at its current position using an $\eta$-move 
in $K$, then the lamplighter makes a $\mu$-move in $H$ according to $\mu$ and, 
finally, the lamplighter switches the lamp at its final position using 
an $\eta$-move in $K$. Each of these steps are performed independently 
of each others. See, e.g., \cite{Pittet2002,SCZ-dv1} for more details. When we write $\eta*\mu*\eta$ in 
$K\wr H$, we identify $\eta$ with the probability measure on $K\wr H$ with is equal to $\eta$ on the copy of $K$ above the identity of $H$ and vanishes everywhere else, and we identify $\mu$ with the 
a probability measure on $K\wr H$ supported on the obvious copy of $H$ in $K\wr H$.

Thanks to \cite{CGP,Erschler2006,Pittet2002,SCZ-dv1}, quite a lot is known about 
the random walk invariant $\Phi_{K\wr H}$.  Further, the results stated in Theorems 
\ref{th-freesol}-\ref{th-NNN} can in fact be 
rephrased as stating that 
$$\Phi_{\Gamma_2(N)}\simeq \Phi_{\mathbb Z^a\wr \Gamma_1(N)}$$
for some/any integer $a\ge 1$. It is relevant to note here that for 
$\Gamma$ of polynomial volume growth of degree $D>0$ or $\Gamma$ infinite 
polycyclic (and in many other cases as well),  whe have 
$\Phi_{\mathbb Z^a\wr \Gamma} \simeq \Phi_{\mathbb Z^b\wr \Gamma}$ 
for any integers $a,b\ge 1$.  Indeed, the proofs of Theorems 
\ref{th-freesol}--\ref{th-NNN}--\ref{th-freesola} make use of the Magnus
embedding which provides  us with an injective homomorphism 
$\bar{\psi}: \Gamma_2(N) \hookrightarrow \mathbb Z^r\wr \Gamma_1(N)$.
This embedding is use to prove a lower bound of the type
$$\Phi_{\Gamma_2(N)}(n)\ge c\Phi_{\mathbb Z^r\wr \Gamma_1(N)}(kn)$$
and an upper bound that can be stated as
$$ \Phi_{\Gamma_2(N)}(Cn)\le C\Phi_{\mathbb Z \wr \overline{\Gamma}}(n)$$
where $\overline{\Gamma}< \Gamma_1(N)$ is a subgroup which has 
a similar structure as $\Gamma_1(N)$. For instance, in the easiest  cases including when $\Gamma_1(N)$ is nilpotent, $\overline{\Gamma}$ is a finite index subgroup of 
$\Gamma_1(N)$.
The fact that the wreath product is taken with $\mathbb Z^r$ in the lower 
bound and with $\mathbb Z$ in the upper bound is not a typo. 
It reflects the nature of the arguments used for the proof. Hence, the fact 
that the lower and upper bounds that are produced by our arguments 
match up depends on the property that, under proper hypotheses on 
$\overline{\Gamma}< \Gamma_1(N)$ and $\Gamma_1(N)$,
$$\Phi_{\mathbb Z^a\wr \Gamma_1(N)} \simeq \Phi_{\mathbb Z^b\wr \overline{\Gamma}}$$ 
for any pair of positive integers $a,b$.

\subsection{A short guide}
Section 2 of the paper is devoted to the algebraic structure of the group 
$\Gamma_2(N)=\mathbf{F}_r/[N,N]$.  It describes the Magnus embedding 
as well as the  interpretation of $\Gamma_2(N)$ 
in terms of flows on $\Gamma_1(N)$. See 
\cite{DLS,Myasnikov2010,Vershik2000}.  
The Magnus embedding and the flow representation play  
key parts in the proofs of our main results.

Section 3 describes two methods to obtain lower bounds on the probability of 
return of certain random walks on $\Gamma_2(N)$.  The first method is based 
on a simple comparison argument and the notion of F\o lner couples 
introduced in \cite{CGP} and already used in \cite{Erschler2006}. This method 
works for symmetric random walks driven by a finitely supported measure.
The second method allows us to treat some measures that are not finitely 
supported, something that is of interest in the spirit of Theorem 
\ref{th-freesola}.

Section 4 focuses on upper bounds for the probability of return. 
This section also makes use of the Magnus embedding, but in a somewhat 
more subtle way. We introduce the notion of exclusive pair. These pairs are made of a subgroup 
$\Gamma$ of $\Gamma_2(N)$ and an element $\boldsymbol \rho$ in the free group $\mathbf F_r$ 
that projects to a cycle on $\Gamma_1(N)$ with the property that the traces of $\Gamma$ and 
$\boldsymbol \rho$ on $\Gamma_1(N)$ have, in a sense, minimal interaction. See Definition 
\ref{def-exclu}. Every upper bound we obtain is proved using this notion.

Section 5 presents a variety of applications of the results obtained in Sections
3 and 4. In particular, the statement regarding $\Phi_{\mathbf{S}_{2,r}}$
as well as Theorems \ref{th-NNN}--\ref{th-freesola} and assorted results
are proved in Section 5.

Section 6 is devoted to the result concerning $\mathbf{S}_{d,r}$, $d\ge 3$. 
Both the lower bound and the upper bound methods are re-examined to allow  
iteration of the procedure.

Throughout this work, we will have to distinguish 
between convolutions in different groups. We will use $*$ 
to denote either convolution on a generic group $G$ 
(when no confusion can possibly arise) or, more specifically, 
convolution on $\Gamma_2(N)$. When $*$ is used to denote convolution 
on $\Gamma_2(N)$, we use $e_*$ to denote the identity element in $\Gamma_2(N)$.
We will use $\star$ to denote convolution on various wreath products such as
$\mathbb Z^r\wr \Gamma_1(N)$. When this notation is used, $e_\star$ will 
denote the identity element in the corresponding group. When necessary, 
we will decorate $\star$ with a subscript to distinguish between 
different wreath poducts. So, if $\mu$ is a probability measure on $\Gamma_2(N)$ and $\phi$ a probability measure on $\mathbb Z^r\wr \Gamma_1(N)$, we will write
$\mu^{*n}(e_*)=\phi^{\star n}(e_\star)$ to indicate that the $n$-fold convolution of $\mu$ on $\Gamma_2(N)$ evaluated at the identity element of $\Gamma_2(N)$ is 
equal to the $n$-fold convolution of $\phi$ on $\mathbb Z\wr \Gamma_1(N)$ evaluated at the identity element of $\mathbb Z\wr \Gamma_1(N)$.

\section{$\Gamma_2(N)$ and the Magnus embedding} \setcounter{equation}{0}

This work is concerned with random walks on the groups 
$\Gamma_\ell(N)=\mathbf F_r/N^{(\ell)}$ where $\mathbf F_r$ is the free group
on $r$ generators and $N$ is a normal subgroup of $\mathbf F_r$. 
In fact, it is best to think of $\Gamma_\ell(N)$ as a marked group, 
that is, a group equipped with a generating tuple. 
In the case of $\Gamma_\ell(N)$, the generating $r$-tuple is always 
provided by the images of the free generators of $\mathbf F_r$.  Ideally,
one would like to obtain results based on hypotheses on the nature of 
$\Gamma_1(N)$ viewed as an unmarked group. 
However, as pointed out in Remark \ref{needN} below,
the unmarked group $\Gamma_1(N)$ is not enough to determine either 
$\Gamma_2(N)$ or the random walk invariant $\Phi_{\Gamma_2(N)}$.
That is, in general, one needs information about the pair $(\mathbf F_r,N)$ 
itself to obtain precise information about $\Phi_{\Gamma_2(N)}$. Note however 
that when $\Gamma_1(N)$ is nilpotent with volume growth of degree at least $2$,
Theorem \ref{th-NNN} provides a result that does not require further 
information on $N$.

\subsection{The Magnus embedding} \label{sec-Magnus}
In 1939, Magnus \cite{Magnus} introduced an embedding of  $
\Gamma_2(N)=\mathbf{F}_{r}/[N,N]$ into a matrix group with coefficients 
in a module over  $\mathbb Z(\Gamma_1(N))=\mathbb{Z}(%
\mathbf{F}_{r}/N)$. In particular, the Magnus embedding is used to embed
free solvable groups into certain wreath products. 

Let $\mathbf F_r$ be the free group on the generators 
$\mathbf s_i$, $1\le i\le r$. 
Let $N$ be a normal subgroup of $%
\mathbf{F}_r$ and let $\pi=\pi_N$ and $\pi_2=\pi_{2,N}$ 
be the canonical projections
$$\pi: \mathbf F_r\ra \mathbf F_r/N=\Gamma_1(N),\;\;\pi_2:\mathbf 
F_r\ra \mathbf F_r/[N,N]=\Gamma_2(N).$$ 
We also let
$$\bar{\pi}: \Gamma_2(N)\ra \Gamma_1(N)$$
the projection from $\Gamma_2(N)$ onto $\Gamma_1(N)$, 
whose kernel can be identified with $N/[N,N]$, has the property
that $\pi= \bar{\pi}\circ \pi_2$.
Set
$$s_i=\pi_2(\mathbf s_i),\;\;
\bar{s}_i=\pi(\mathbf s_i)=\bar{\pi}(s_i).$$ 
When it is necessary to distinguish between the identity element in $e\in \Gamma_2(N)$ and 
the identity element in $\Gamma_1(N)$, we write $\bar{e}$ for the latter.

Let $\mathbb{Z}(\mathbf{F}_r)$ be the integral group
ring of the free group $\mathbf{F}_r$. By extension and with some abuse 
of notation, let $\pi $ denote also the ring homomorphism 
\[
\pi : 
\mathbb{Z}
(\mathbf{F_r)}\rightarrow 
\mathbb{Z}
(\mathbf{F}_r/N)
\]%
determined by $\pi (\mathbf{s}_{i})= \bar{s}_{i}$, $1\le i\le r$.

Let $\Omega $ be the free left $
\mathbb{Z}
(\mathbf{F}_r/N)$-module of rank $r$ with basis $(\lambda _{\mathbf{s}%
_{i}})_1^r$ and set
\[
M=\left[ 
\begin{array}{cc}
\mathbf{F}_r/N & \Omega \\ 
0 & 1%
\end{array}%
\right] 
\]%
which is a subgroup of the group of the $2\times 2$ upper-triangular matrices 
over $\Omega .$ The
map 
\begin{equation}\label{defpsi}
\psi (\mathbf{s}_{i})=\left[ 
\begin{array}{cc}
\pi (\mathbf{s}_{i}) & \lambda _{\mathbf{s}_{i}} \\ 
0 & 1%
\end{array}%
\right] 
\end{equation}
extends to a homomorphism $\psi $ of $\mathbf{F}_r$ into $M$. We denote by $%
a(\mathbf u),$ $\mathbf u\in \mathbf{F}_r$, 
the $(1,2)$-entry of the matrix $\psi (\mathbf u),$ that is%
\begin{equation}\label{defa}
\psi (\mathbf u)=\left[ 
\begin{array}{cc}
\pi (\mathbf u) & a(\mathbf u) \\ 
0 & 1%
\end{array}%
\right] . 
\end{equation}

\begin{theo}[Magnus \protect\cite{Magnus}]
The kernel of the homomorphism $\psi :\mathbf{F}_r\rightarrow M$ defined as
above is 
\[
\ker (\psi )=[N,N]. 
\]%
Therefore $\psi $ induces a monomorphism%
\[
\bar{\psi} :\mathbf{F}_r/[N,N]\hookrightarrow M. 
\]
It follows that $\mathbf{F}_r/[N,N]$\bigskip\ is isomorphic to the subgroup
of $M$ generated by 
\[
\left[ 
\begin{array}{cc}
\pi (\mathbf{s}_{i}) & \text{ }\lambda _{\mathbf{s}_{i}} \\ 
0 & \text{ }1%
\end{array}%
\right], \;i=1,\dots,r . 
\]
\end{theo}
\begin{rem} For $g\in \mathbf F_r/[N,N]$, we write \begin{equation}\label{defbara}
\bar{\psi} (g)=\left[ 
\begin{array}{cc}
\bar{\pi} (g) & \bar{a}(g) \\ 
0 & 1%
\end{array}%
\right] 
\end{equation}
where $\bar{a}(\pi_2(\mathbf u))=a(\mathbf u)$, $\mathbf u\in \mathbf F_r$.
\end{rem}

\begin{rem}
The free left $
\mathbb{Z}
(\mathbf{F}_r/N)$-module $\Omega $ with basis $\{\lambda _{\mathbf{s}%
_{i}}\}_{1\leq i\leq d}$ is isomorphic to the direct sum $\sum_{x\in 
\mathbf{F}_r/N}(\mathbb{Z}
^{r})_{x}.$ More precisely, if we regard the elements in 
$\sum_{x\in \mathbf{F}_r/N}(
\mathbb{Z}
^{r})_{x}$ as functions $f=(f_{1},..,f_{r}):\mathbf{F}_r/N\rightarrow 
\mathbb{Z}
^{r}$ with finite support, the map%
\begin{eqnarray*}
\sum_{x\in \mathbf{F}_r/N}(
\mathbb{Z}
^{r})_{x} &\rightarrow &\Omega : \\
f &\mapsto &\left( \sum_{x\in \mathbf{F}_r/N}f_{1}(x)x\right) \lambda _{%
\mathbf{s}_{1}}+...+\left( \sum_{x\in \mathbf{F}/N}f_{r}(x)x\right) \lambda
_{\mathbf{s}_{r}}
\end{eqnarray*}%
is a left $\mathbb Z(\mathbf{F}_r/N)$-module isomorphism. 
We will identity $\Omega $ with 
$\sum_{x\in \mathbf{F}_r/N}(
\mathbb{Z}
^{r})_{x}.$
Using the above interpretation, one can restate the Magnus embedding 
theorem as an
injection from $\mathbf{F}_r/[N,N]$ into the wreath product $
\mathbb{Z}
^{r}\wr (\mathbf{F}_r/N).$
\end{rem}

The entry $a(g)\in \Omega $ under the Magnus embedding is given by 
Fox derivatives which we briefly review. Let $G$ be a group and $
\mathbb{Z}
(G)$ be its integral group ring. Let  $M$  be a left $\mathbb Z(G)$-module. An additive
map $d:
\mathbb{Z}
(G)\rightarrow M$ is called a \textit{left derivation} if for all $x,y\in G,$%
\[
d(xy)=xd(y)+d(x). 
\]
As a consequence of the definition,  we have $d(e) =0$ and 
$d(g^{-1}) =-g^{-1}d(g)$.

For the following two theorems of Fox, we refer the reader to the discussion in 
\cite[Sect. 2.3]{Myasnikov2010} and the references given there.
\begin{theo}[Fox]
Let $\mathbf{F}_r$ be the free group on $r$ generators $\mathbf{s}_{i}$, $%
1\leq i\leq r$. For each $i$, there is a unique left derivation 
\[
\partial _{\mathbf{s}_{i}}:
\mathbb{Z}
(\mathbf{F}_r)\rightarrow 
\mathbb{Z}
(\mathbf{F}_r) 
\]
satisfying
\[
\partial _{\mathbf{s}_{i}}(\mathbf{s}_{j})=\left\{ 
\begin{array}{ccc}
1 & \text{ \ if \ } & i=j \\ 
0 & \text{ \ if \ } & i\neq j .%
\end{array}%
\right. 
\]
Further, if $N$ is a normal subgroup of $\mathbf F_r$, then
$\pi(\partial_{\mathbf s_i}\mathbf u)= 0$ in $\mathbb Z(\mathbf F_r/N)$ 
for all $1\le i\le r$ if and only if
$\mathbf u\in [N,N]$.
\end{theo}

\begin{exa}
For $\mathbf{g}=\mathbf{s}_{i_{1}}^{\varepsilon _{1}}...\mathbf{s}%
_{i_{n}}^{\varepsilon _{n}}$, $\varepsilon _{j}\in \{\pm 1\},$%
\begin{eqnarray*}
\partial _{\mathbf{s}_{i}}(\mathbf{g}) &=&\sum_{j=1}^{n}\mathbf{s}%
_{i_{1}}^{\varepsilon _{1}}...\mathbf{s}_{i_{j-1}}^{\varepsilon
_{j-1}}\partial _{\mathbf{s}_{i}}(\mathbf{s}_{i_{j}}^{\varepsilon _{j}}) \\
&=&\sum_{j:i_{j}=i,\varepsilon _{j}=1}\mathbf{s}_{i_{1}}^{\varepsilon
_{1}}...\mathbf{s}_{i_{j-1}}^{\varepsilon
_{j-1}}-\sum_{j:i_{j}=i,\varepsilon _{j}=-1}\mathbf{s}_{i_{1}}^{\varepsilon
_{1}}...\mathbf{s}_{i_{j-1}}^{\varepsilon _{j-1}}\mathbf{s}%
_{i_{j}}^{\varepsilon _{j}}.
\end{eqnarray*}
\end{exa}

\begin{theo}[Fox]
The Magnus embedding 
\[
\bar{\psi} :\mathbf{F}_r/[N,N]\hookrightarrow M 
\]%
is given by%
\begin{equation}
\bar{\psi} (g)=\left[ 
\begin{array}{cc}
\bar{\pi} (g) & \text{ }\sum_{i=1}^{r}\pi (\partial _{\mathbf{s}_{i}}\mathbf{g}%
)\lambda _{\mathbf{s}_{i}} \\ 
0 & \text{ }1%
\end{array}%
\right]   \label{Magnus}
\end{equation}
where $\mathbf g\in \mathbf F_r$ is any element such that $\pi_2(\mathbf g)=g$.
\end{theo}
\begin{exa}
In the special case that $N=[\mathbf{F}_r,\mathbf{F}_r]$, we have $\mathbf{F}_r
/N\simeq 
\mathbb{Z}
^{r}$ and $
\mathbb{Z}
(\mathbf{F}_r/N)$ is the integral group ring over the free abelian group $
\mathbb{Z}
^{r}.$ The integral group ring $
\mathbb{Z}(
\mathbb{Z}
^{r})$ is quite similar to the multivariate polynomial ring with integer
coefficients, except that we allow negative powers like $%
Z_{1}^{-3}Z_{2}...Z_{r}^{-5}.$ The monomials $%
\{Z_{1}^{x_{1}}Z_{2}^{x_{2}}...Z_{r}^{x_{r}}:x\in
\mathbb{Z}
^{r}\}$ are $
\mathbb{Z}
$-linear independent in $
\mathbb{Z}(
\mathbb{Z}
^{r}).$
\end{exa}

\subsection{Interpretation in terms of flows} \label{sec-flow}

Following \cite{DLS,Myasnikov2010,Vershik2000}, one can also think of 
elements of $\Gamma_2(N)=\mathbf{F}_r/[N,N]$ in terms of flows on  
the (labeled) Cayley graph of $\Gamma_1(N)=\mathbf{F}_r/N$.   
To be precise, Let $\mathbf s_1,\dots,\mathbf s_k$ be the generators of 
$\mathbf F_r$ and $\bar{s}_1,\dots,\bar{s}_k$ their images in $\Gamma_1(N)$. 
The Cayley graph of  
$\Gamma_1(N)$ is the marked graph
with vertex set $V=\Gamma_1(N)$ and marked edge set 
$\mathfrak E\subset V\times V\times \{\mathbf s_1,\dots,\mathbf s_k\}$ 
where $(x,y, \mathbf s_i)\in \mathfrak E $ if and only if 
$y=x\bar{s}_i$ in $\Gamma_1(N)$. 
Note that each edge $\mathfrak e=(x,y,\mathbf s_i)$ as an origin 
$o(\mathfrak e)=x$, 
an end (or terminus) $t(\mathfrak e)=y$ and a label or mark $\mathbf s_i$.

Given a function $\mathfrak f$ on the edge set $\mathfrak E$ 
and a vertex $v\in V$, define the net flow $\mathfrak f^\ast(v)$ 
of $\mathfrak f$ at $v$ by
$$\mathfrak f^*(v)= \sum_{o(\mathfrak e)=v}f(\mathfrak e)-\sum_{t(\mathfrak e)
=v}f(\mathfrak e).$$
A \textit{flow} (or $\mathbb Z$-flow) with source $s$ and
sink $t$ is a function 
$\mathfrak f:\mathfrak  E\ra \mathbb Z$ 
such that 
\begin{eqnarray*} &&
\forall\, v\in V\backslash \{s,t\},\;\;\mathfrak{f}^{\ast }(v)=0 , \\
&&\mathfrak f^\ast(s)=1,\;\mathfrak f^\ast(t)=-1 . \nonumber
\end{eqnarray*}
If  $\mathfrak f^\ast(v)=0$ holds for all
$v\in V $, we say that $\mathfrak f$ is a \textit{circulation}.  

For each edge $\mathfrak e=(x,y,\mathbf s_i)$, introduce its formal inverse 
$(y,x,\mathbf s_i^{-1})$ and let $\mathfrak E^*$ be the set of all edges and 
their formal inverses. A finite path on the Cayley graph of $\Gamma_1(N)$ 
is a finite sequence $p=(\mathfrak e_1,\dots,\mathfrak e_\ell)$ of edges in 
$\mathfrak E^*$ so that the origin
of $\mathfrak e_{i+1}$ is the terminus of $\mathfrak e_i$. 
We call $o(\mathfrak e_1)$ (resp. $t(\mathfrak e_\ell)$)
the origin (resp. terminus) 
of the path $p$ and denote it by $o(p)$ (resp. $t(p)$).
Note that reading the labels 
along the edge of a path determines a word in the generators of $\mathbf F_r$ 
and that, conversely, any finite word $\omega$ in the generators of 
$\mathbf F_r$ determines a path $p_\omega$
starting at the indentity element in $\Gamma_1(N)$.  
 
A (finite) path $p$ determines a flow $\mathfrak f_p$ with source $o(p)$ 
and sink $t(p)$ by setting $\mathfrak f_p(e)$ to be the algebraic number of time the edge $\mathfrak e\in \mathfrak E$ is crossed positively or negatively along $p$. Here, 
the edge $\mathfrak e=(x,y,\mathbf s_\alpha)\in \mathfrak E$ 
is crossed positively at the $i$-step
along $p$ if $\mathfrak e_i=(x,y,\mathbf s_\alpha)$. It is crossed negatively if 
$\mathfrak e_i=(y,x,\mathbf s^{-1}_\alpha)$.  We note that 
$\mathfrak f_p$ has finite support and that either $o(p)=t(p)$ and 
$\mathfrak f_ p$ is a circulation or $o(p)\neq t(p)$ and 
$\mathfrak f_p^*(o(p))=1$,
$\mathfrak f_p^*(t(p))=-1$.

Given a word $\omega =\mathbf{s}_{i_{1}}^{\varepsilon _{1}}...\mathbf{s}%
_{i_{n}}^{\varepsilon _{n}}$ in the generators of $\mathbf{F}_r$, let 
$\mathfrak{f}_{\omega }$ denote the flow function on the Cayley graph 
of $\Gamma_1(N)$ defined by the corresponding path starting at 
the identity element in $\Gamma_1(N)$. We note that it is obvious from 
the definition that  $\mathfrak f_\omega=\mathfrak f_{\omega'}$ if $\omega'$ is
the reduced word in $\mathbf F_r$ associated with $\omega$.

\begin{theo}[{\protect\cite[Theorem 2.7]{Myasnikov2010}}] \label{th-flowdescr}
Two elements $\mathbf{u},\mathbf{v}\in \mathbf{F}_r$ project to the
same element in $\Gamma_2(N)=\mathbf{F}_r/[N,N]$ if and only if they 
induce the same
flow on $\Gamma_1(N)=\mathbf{F}_r/N.$ In other words,%
\[
\mathbf{u}\equiv \mathbf{v}\text{ }\mbox{\em  mod }[N,N]\;\Longleftrightarrow\; 
\mathfrak{f}_{\mathbf{u}}=\mathfrak{f}_{\mathbf{v}}. 
\]
\end{theo}

This theorem shows that an element $g\in \Gamma_2(N)$
corresponds to a unique flow $\mathfrak{f}_{\omega }$ on $\mathbf{F}_r/N,$
defined by the path $p_{\omega }$ associated with any  word 
$\omega \in \mathbf{F}_r$ such that $%
\omega $ projects to $g$ in $\Gamma_2(N)$. For $g\in \Gamma_2(N)$, 
$\mathfrak{f}_{g}:=\mathfrak{f}_{\omega }$ is well defined (i.e., is independent 
of the word $\omega$ projecting to $g$, and we call $\mathfrak f _g$
the flow of $g$. Hence, in a certain sense, we can  regard
elements of $\Gamma_2(N)$ as flows on $\Gamma_1(N)$. In fact, the flow $\mathfrak f_\omega$ is directly related to the description of the image of the 
element $g=\omega \mbox{ mod } [N,N]$ under the
Magnus embedding through the following geometric interpretation of 
Fox derivatives.

\begin{lem}[{\protect\cite[Lemma 2.6]{Myasnikov2010}}]
\label{FlowMagnus} Let $\omega \in \mathbf{F}_r$, then for any $g\in \mathbf{F}_r%
/N$ and $\mathbf{s}_{i}$, the value of $\mathfrak{f}_{\omega }$ on the edge $%
(g,gs_{i},\mathbf s_i)$, is equal to coefficient in front of $g$ 
in the Fox derivative $%
\pi (\partial _{\mathbf{s}_{i}}\omega )\in 
\mathbb{Z}
(\mathbf{F}/N),$ i.e.%
\begin{equation}
\pi (\partial _{\mathbf{s}_{i}}\omega )=\sum_{g\in \mathbf{F}/N}\mathfrak{f}%
_{\omega }((g,gs_{i},\mathbf s_i))g.  \label{FlowCoefficient}
\end{equation}
\end{lem}

There is also a characterization of geodesics on $\Gamma_2(N)$ in
terms of flows (see \cite[Theorem 2.11]{Myasnikov2010}) which is closely 
related to the description of geodesics on wreath products. 
See \cite[Theorem 2.6]{Sale} where it is proved that the Magnus embedding is 
bi-Lipschitz with small explicit universal distortion.

\begin{rem}\label{needN} 
In \cite{Erschler2004}, it is asserted that the group $\Gamma_2(N)$
depends only of $\Gamma_1(N)$ (in \cite{Erschler2004}, $\Gamma_1(N)$ is denoted by $A$ and $\Gamma_2(N)$ by $C_A$). This assertion is correct only if one 
interprets $\Gamma_1(N)$ as a marked group, i.e., if information about 
$\pi: \mathbf F_r \ra \Gamma_1(N)$ is retained. Indeed,
$\Gamma_2(N)$ depends in some essential ways of the choice of the presentation 
$\Gamma_1(N)=\mathbf F_r/N$. 
We illustrate this fact by two examples that are very good to keep in mind.
\end{rem}

\begin{exa}
Consider two presentations of $\mathbb{Z}$, namely, 
$\mathbb{Z}=\mathbf{F}_{1}$ and $%
\mathbb{Z=}\left\langle a,b|b\right\rangle $. In the first presentation, the
kernel $N_{1}$ is trivial, therefore $\mathbf{F}_{1}/[N_{1},N_1]\simeq 
\mathbb{Z}$. In the second presentation, the kernel $N_{2}$ is the normal
closure of $\left\langle b\right\rangle $ in the free group $\mathbf{F}_{2}$
on generators $a,b$. Hence, $N_{2}$ is generated by 
$\{a^{i}ba^{-i},i\in \mathbb{Z}%
\}$. We can then write down a presentation of $\mathbf{F}_{2}/[N_{2},N_2]
$ in the form
\[
\mathbf{F}_{2}/[N_{2},N_2]=\left\langle
a,b|[a^{i}ba^{-i},a^{j}ba^{-j}],i,j\in \mathbb{Z}\right\rangle . 
\]%
This is, actually,  a presentation of the wreath product $\mathbb{Z\wr Z}$.
Therefore $\mathbf{F}_{2}/N_{2}^{\prime }\simeq \mathbb{Z\wr Z}$. 
We encourage the reader to recognize the structure of both 
$\mathbf{F}_{1}/[N_{1},N_1]\simeq \mathbb Z$ and 
$\mathbf{F}_{2}/[N_{2},N_2]\simeq \mathbb{Z\wr Z}$ using flows on the 
labeled Cayley graphs associated with $\mathbf F_1/N_1$ and $\mathbf F_2/N_2$. 
The Cayley graph of $\mathbf F_2/N_2$ is the usual line graph of $\mathbb Z$ 
decorated with  an oriented loop at each vertex. In the flow representation of an element of 
$\mathbf F_2/[N_2,N_2]$, the algebraic number of times the 
flow goes around each of these loops is recorded thereby creating 
the wreath product struture of $\mathbb Z\wr \mathbb Z$. 
\end{exa}

\begin{exa}Consider the following two presentations of $\mathbb{Z}^{2}$, 
\begin{eqnarray*}
\mathbb{Z}^{2} &=&\left\langle a,b|[a,b]\right\rangle  \\
\mathbb{Z}^{2} &\mathbb{=}&\left\langle a,b,c|[a,b],c=ab\right\rangle .
\end{eqnarray*}%
Call $N_1\subset \mathbf F_2$ and $N_2\subset \mathbf F_3$ be the 
associated normal subgroups.
We claim that  $\mathbf{F}_{2}/[N_{1},N_1]$ is a proper quotient  of $%
\mathbf{F}_{3}/[N_{2},N_2]$. Let $\theta :\mathbf{F}_{3}\rightarrow 
\mathbf{F}_{2}$ be the homomorphism determined by $\theta (a)=a$, $\theta
(b)=b$, $\theta (c)=ab$. Obviously, $N_{2}=\theta ^{-1}(N_{1})$,
$[N_{2},N_2]\subset \theta ^{-1}([N_{1},N_1])$, and  $\theta $ 
induces a surjective
homomorphism $\theta ^{\prime }:\mathbf{F}_{3}/[N_{2},N_2]\rightarrow 
\mathbf{F}_{2}/[N_{1},N_1]$. The element $abc^{-1}$ is nontrivial in $%
\mathbf{F}_{3}/[N_{2},N_2]$, but $\theta ^{\prime }(abc^{-1})=e$.  A
Hopfian group is a group that cannot be isomorphic to a proper quotient 
of itself. Finitely generated metabelian groups are Hopfian. Hence 
$\mathbf{F}_{2}/[N_{1},N_1]$ is not isomorphic to  $%
\mathbf{F}_{3}/[N_{2},N_2]$.
\end{exa}

\section{Return probability lower bounds} \setcounter{equation}{0}
\label{sec-low}

\subsection{Measures supported by the powers of the generators}\label{sec-mes}
The group $\Gamma_2(N)=\mathbf F_r/[N,N]$ commes equiped with the generators $(s_i)_1^r$ which are the images of the generators $(\mathbf s_i)_1^r$ 
of $\mathbf F_r$. Accordingly, we consider a special class of symmetric 
random walks defined as follows. Given probability measures 
$p_i$, $1\le i\le r$ on $\mathbb Z$, we define a probability measure 
$\boldsymbol \mu$ on $\mathbf F_r$ by
\begin{equation} 
\forall\; \mathbf g \in \mathbf F_r,\;\;\boldsymbol{\mu }(\mathbf{g})=\sum_{i=1}^{r}\frac{1}{r}\sum_{m\in 
\mathbb{Z}
}p _{i}(m)\mathbf{1}_{\{\mathbf{s}_{i}^{m}\}}(\mathbf{g}).
\label{pushforward}
\end{equation}
This probability measure induces pushforward measures  $\bar{\mu}$ and $\mu$
on  $\Gamma_1(N)=
\mathbf F_r/N$ and $\Gamma_2(N)=\mathbf F_r/[N,N]$, namely, 
\begin{equation}\left\{\begin{array}{l}
\forall\,\bar{g}\in \Gamma_1(N),\;\;\bar{\mu}(\bar{g})= 
\boldsymbol{\mu}(\pi^{-1}(\bar{g}))\\
\forall\,g\in \Gamma_2(N),\;\;\mu(g)= 
\boldsymbol{\mu}(\pi_2^{-1}(g)). \end{array}\right.\label{defmu}
\end{equation}  

In fact, we will mainly consider two cases. In the first case, each $p_i$
is the measure of the lazy random walk on $\mathbb Z$, 
that is $p_i(0)=1/2$, $p_i(\pm 1)=1/4$. In this case, $\boldsymbol{\mu}$ is 
the measure of the lazy simple random walk on $\mathbf F_r$, that is,
\begin{equation}\label{lazysr}
\boldsymbol{\mu }(e)=1/2,\;\;\boldsymbol{\mu}(\mathbf s_i^{\pm 1})=1/4r.
\end{equation}
The second case can be viewed as a generalization of the first. 
Let $a=(\alpha_1)_1^r\in (0,\infty]^r$ be a $r$-tuple of extended positive reals.
For each $i$, consider the symmetric probability measure $p_{\alpha_i}$
on $\mathbb Z$ with $p_{\alpha_i}(m)=c_i(1+|m|)^{-1-\alpha_i}$
(if $\alpha_i=\infty$, set $p_\infty(0)=1/2$, $p_\infty(\pm 1)=1/4$). 
Let
$\boldsymbol{\mu}_a$ be the measure on $\mathbf F_r$ obtained by 
setting $p_i=p_{\alpha_i}$ in (\ref{pushforward}). When $a$ is such that 
$\alpha_i=\infty$ for all $i$ we recover (\ref{lazysr}).  
In particular, starting with (\ref{lazysr}),  
$\mu$ is given by
$$ \forall\,g\in \Gamma_2(N),\;\;\;
\mu(g)= \frac{1}{2} \mathbf 1_{g}(e) +\frac{1}{4r}\sum_1^r\mathbf 1_{s_i}(g).$$
The formula for $\bar{\mu}$ is exactly similar.  
For any fixed $a\in (0,\infty]^r$, we let $\mu_a$ and $\bar{\mu}_a$ 
be the pushforward of $\boldsymbol{\mu}_a$ on $\Gamma_2(N)$ and $\Gamma_1(N)$, respectively.

\subsection{Lower bound for simple random walk}\label{sec-low1}
In this section, we explain how, in the case of the lazy simple random walk 
measure $\mu$ on $\Gamma_2(N)$ associated with $\boldsymbol{\mu}$ 
at (\ref{lazysr}), 
one can obtained lower bounds for the probability of return 
$\mu^{(2n)}(e)$ by using well-known arguments 
and the notion of F\o lner couples introduced in \cite{CGP}. 

\begin{defin}[See {\cite[Definition 4.7]{CGP}} and 
{\cite[Proposition 2]{Erschler2006}}] \label{folner}
Let $G$ be a finitely generated group equipped with a finite 
symmetric genereting set $T$ and the associated word distance $d$.
Let $\mathcal V$ be a positive increasing function on $[1,\infty)$
whose inverse is defined on $[\mathcal V(1),\infty)$. We say that a sequence 
of pairs of nonempty sets $((\Omega_k,\Omega'_k))_1^\infty$ is a
a sequence of F\o lner couples adapted to $\mathcal V$ if 
\begin{enumerate}
\item  $\Omega'_k\subset \Omega_k$, 
$\#\Omega'_k \ge c_0 \#\Omega_k$, $d(\Omega'_k,\Omega_k^c)\ge c_0k$.
\item 
$v_k=\#\Omega_k \nearrow \infty$ and  
$v_k\le \mathcal V(k)$. 
\end{enumerate}
\end{defin} 
Let $\nu$ be a symmetric finitely suppported measure on $G$ and $\lambda_1(\nu,\Omega)$ be the lowest Dirichlet eignevalue in $\Omega$
for the convolution by $\delta_e-\nu$, namely,
$$\lambda_1(\nu,\Omega)=\inf\left\{\sum_{x,y}|f(xy)-f(x)|^2\nu(y): 
\mbox{supp}(f)\in \Omega, \sum |f|^2=1\right\}.$$ 
If $(\Omega_k,\Omega_k')$ is a pair satisfying the first condition in 
Definition \ref{folner} then plugging $f=d(\cdot,\Omega_k^c)$ in the definition of of $\lambda_1(\nu,\Omega_k)$ immeditely gives
$\lambda_1(\nu,\Omega_k)\le \frac{C}{k^2}.$

Given a function $\mathcal V$ as in Definition \ref{folner}, let $\gamma$
be defined implicitely by
\begin{equation}\label{gamma}
\int_{\mathcal V(1)}^{\gamma(t)}([\mathcal V^{-1}(s)]^2\frac{ds}{s}=t.
\end{equation}
This is the same as stating that $\gamma$ is a solution of the 
differential equation 
\begin{equation}\label{gamma'}
\frac{\gamma'}{\gamma}=\frac{1}{[\mathcal V ^{-1}\circ \gamma]^2},\;\;\gamma(0)=\mathcal V(1).\end{equation}
Following \cite{Erschler2006}, we say that $\gamma$ is $\delta$-regular  if  $\gamma'(s)/\gamma(s)\ge \delta \gamma'(t)/\gamma(t)$ for all $s,t$ with 
$0<t<s<2t$.

With this notation, Erschler \cite[Proposition 2]{Erschler2006} gives a 
modified version of \cite[Theorem 4.7]{CGP} which contains the following 
statement.
\begin{pro} \label{prob-FC}
If the group $G$ admits a sequence of F\o lner couples adapted
to the function $\mathcal V$ as in {\em
Definition \ref{folner}} and  the function $\gamma$ associated to $\mathcal V$ by {\em (\ref{gamma})} is $\delta$-regular for some $\delta>0$ then there exist 
$c,C\in (0,\infty)$ such that
$$\Phi_G(n)\ge \frac{c}{\gamma (Cn)}.$$
\end{pro}
A key aspect of this statement is that it allows for very fast growing 
$\mathcal V$ as long as one can check that $\gamma$ is 
$\delta$-regular. Erschler \cite{Erschler2006} gives a variety of examples
showing how this works in practice but it seems
worth explaining why the $\delta$-regularity of $\gamma$ is a relatively 
mild assumption.   
Suppose first that $\mathcal V$ is regularly varying of positive finite index.
Then the same is true for $\mathcal V^{-1}$ and 
$\int_{\mathcal V(1)}^T\mathcal V^{-1}(s)^2\frac{ds}{s}\sim 
c \mathcal V^{-1}(T)^2$. In this case, it follows 
from (\ref{gamma'}) that $\gamma'(s)/\gamma(s)\simeq  1/s$.
If instead we assume that $\log \mathcal V$ is of regular variation 
of positive index
(resp. rapid variation) then
$\mathcal V^{-1}\circ \exp$ is of regular variation of positve index 
(resp. slow variation) and we can show that
$$\int_{\mathcal V(1)}^T\mathcal V^{-1}(s)^2\frac{ds}{s}\simeq 
\mathcal V^{-1}(T)^2 \log T.$$ In this case, it follows again that
$\gamma$ is $\delta$-regular. All the examples treated in \cite{Erschler2006} 
and in the present paper fall in these categories.

The following proposition regarding wreath products is key.

\begin{pro}[Proof of {\cite[Theorem 2]{Erschler2006}}] \label{wr-FC}
Assume that the group  $G$ is infinite, finitely generated, 
and admits  a sequence of F\o lner couples adapted to the function $\mathcal V$ as in 
{\em Definition \ref{folner}}. Set
\begin{eqnarray*}
\Theta _{k} &=&\{(f,x)\in 
\mathbb{Z}
^{r}\wr G:x\in \Omega _{k},\;\;\text{\em supp}(f)\subset \Omega _{k},\left\vert
f\right\vert _{\infty }\leq k\#\Omega _{k}
\}, \\
\Theta _{k}^{\prime } &=&\{(f,x)\in 
\mathbb{Z}
^{r}\wr G: x\in \Omega _{k}^{\prime },\;\;
\text{\em supp}(f)\subset \Omega _{k},\left\vert
f\right\vert _{\infty }\leq k\#\Omega _{k}-k
\}.
\end{eqnarray*}%
Set $$\mathcal{W}(v):=\exp\left( C \mathcal{V}(v) \log \mathcal V(v)\right).$$
Then $(\Theta_k,\Theta'_k)$ is a sequence of F\o lner couples on $\mathbb Z^r\wr G$ adapted to $\mathcal W$ (for an appropriate choice of the constant $C$).
\end{pro} 
\begin{proof} By construction (and with an obvious choice of generators in 
$\mathbb Z^r\wr G$ based on a given set of generators for $G$), 
the distance  between $\Theta_k'$ and 
$\Theta_k^c$ in $\mathbb Z^r\wr G$ is greater or equal to the minimum of $k$ and the distance between 
$\Omega_k'$ and $\Omega_k^c$ in $G$. Also, we have
$$\#\Theta_k =\# \Omega_k (k\#\Omega_k)^{r\#\Omega_k},\;\;
\#\Theta'_k = \# \Omega'_k (k\#\Omega_k -k)^{r\#\Omega_k}$$
so that $$ \frac{\#\Theta'_k}{\#\Theta_k}\ge  \left(1-(\#\Omega_k)^{-1}\right)^{r\#\Omega_k}\frac{\#\Omega'_k}{\#\Omega_k}\ge \frac{1}{e^r}\frac{\#\Omega'_k}{\#\Omega_k}$$
and
$$\#\Theta_k =\exp\left( \log \#\Omega_k+ r\#\Omega_k(\log \#\Omega_k +\log k)\right)\le \exp\left(C \mathcal V(k)\log \mathcal V(k)\right).$$ 
\end{proof}

\begin{pro}[Computations]\label{pro-comp}
Let $\mathcal V$ be given. 
Define $\mathcal W$ and $\gamma=\gamma_{\mathcal W}$ by
$$\mathcal W=\exp(C\mathcal V \log \mathcal V) \mbox{ and }
\gamma^{-1}(t)=\int _{\mathcal W(1)}^t [\mathcal W^{-1}(s)]^2\frac{ds}{s}.$$
\begin{enumerate}
\item Assume that $\mathcal V(t)\simeq t^D$. Then  we have $$\gamma(t)\simeq
\exp\left( t^{D/(2+D)}[\log t]^{2/(2+D)}\right).$$
\item Assume that  
$\mathcal V(t) \simeq \exp( t^\alpha  \ell(t))$, $\alpha>0$, where 
$\ell(t)$ is slowly varying with $\ell(t^a)\simeq \ell(t)$ 
for any fixed $a>0$. Then $\gamma$ satisfies
$$\gamma(t)\simeq \left( t \left(\frac{\ell(\log t)}{\log t}\right)^{2/\alpha}  \right).$$
\item Assume that  
$\mathcal V(t) \simeq \exp(\ell^{-1} (t))$ where 
$\ell(t)$ is slowly varying with $\ell(t^a)\simeq \ell(t)$ 
for any fixed $a>0$. Then $\gamma$ satisfies
$$\gamma(t)\simeq \left( t/[\ell(\log t)]^2  \right).$$
Note that if $\ell^{-1}(t)=\exp\circ \cdots \circ \exp (t\log t)$ 
with $m$ exponentials then 
$$\ell (t) \simeq \frac{\log _m t}{\log _{m+1}(t)}.$$  
\end{enumerate}
\end{pro}

\begin{theo} \label{theo-fol}
Let $N$ be a normal subgroup of $\mathbf F_r$. Assume that
the group $\Gamma_1(N)=\mathbf F_r/N$ admits a sequence of F\o lner couples adapted to 
the function $\mathcal V$ as in {\em Definition \ref{folner}}. Let
$\mathcal W$ and $\gamma=\gamma_\mathcal W$ be related to $\mathcal V$ as in 
{\em Proposition \ref{pro-comp}}. Then we have
$$\Phi_{\Gamma_2(N)}(n)\ge \frac{c}{\gamma(Cn)}.$$ 
\end{theo}
\begin{proof}By the Magnus embedding, $\Gamma_2(N)$ 
is a subgroup of $\mathbb Z^r\wr \Gamma_1(N)$. By \cite[Theorem 1.3]{Pittet2000},
it follows that $\Phi_{\Gamma_2(N)}\ge \Phi_{\mathbb Z^r\wr \Gamma_1(N)}$.
The conclusion then follows from Propositions \ref{prob-FC}--\ref{wr-FC}.
\end{proof}

\begin{exa} \label{exa-pol} 
Assume $\Gamma_1(N)$ has polynomial volume growth of degree $D$. Then 
$\Phi_{\Gamma_2(N)}(n)\ge \exp\left(-c n^{D/(2+D)}[\log n]^{2/(2+D)}\right).$
\end{exa}

\begin{exa} 
Assume $\Gamma_1(N)$ is either polycyclic or equal to the Baumslag-Solitar group
$\mbox{BS}(1,q)=\langle a,b| a^{-1}ba=b^q\rangle$, or 
equal to the lamplighter group $F\wr \mathbb Z$ with $F$ finite. Then 
$\Phi_{\Gamma_2(N)}(n)\ge \exp\left(-c n /[\log n]^{2}\right).$
\end{exa}

\begin{exa} \label{exa-wr}
Assume $\Gamma_1(N)=F\wr \mathbb Z^D$ with $F$ finite . Then 
$$\Phi_{\Gamma_2(N)}(n)\ge \exp\left(-c n/[\log n]^{2/D}\right).$$ 
If instead $\Gamma_1(N)= \mathbb Z^b\wr \mathbb Z^D$ for some integer $b\ge 1$ then
$$\Phi_{\Gamma_2(N)}(n)\ge 
\exp\left(-c n\left(\frac{\log\log n}{\log n}\right)^{2/D}\right).$$ 
\end{exa}

\subsection{Another lower bound}\label{sec-low2}
The aim of this subsection is to provide lower bounds for the probability of 
return $\mu^{*n}(e_*)$ on $\Gamma_2(N)$ when $\mu$ at (\ref{defmu})
is the pushforward of a 
measure $\boldsymbol{\mu}$ on $\mathbf F_r$ of the form (\ref{pushforward}), 
that is, supported on the powers of the generators $\mathbf s_i$, $1\le i\le r$, 
 possibly with unbounded support. Our
approach is to construct symmetric probability measure $\phi $ on $%
\mathbb{Z}
^{r}\wr \Gamma_1(N)$ such that the return probability $\phi ^{\star n}(%
e_\star)$ of the random walk driven by $\phi $ coincides with $\mu ^{\ast
n}(e_*)$. Please note that we will use the notation $\star$ for convolution on 
the wreath product $\mathbb Z^r\wr \Gamma_1(N)$ and $\ast$ for convolution 
on $\Gamma_2(N)$. We also decorate the identity element $e_*$
of $\Gamma_2(N)$ with a $*$ to distinguish it from the identity element
$e_\star$ of $\mathbb Z^r\wr \Gamma_1(N)$.  Recall that the 
identity element of $\Gamma_1(N)$ is denoted by $\bar{e}$.
We will use $(\epsilon_i)_1^r$ for the 
canonical basis of $\mathbb Z^r$.

Fix $r$ symmetric probability measures $p_i$, $1\le i\le r$ on $\mathbb Z$. 
Recall that, by definition, $\mu$ is the pushforward of $\boldsymbol{\mu}$, 
the probability measure on $\mathbf F_r$ which gives probability 
$r^{-1}p_i(n)$ to $\mathbf s_i^n$, $1\le i\le r$, $n\in \mathbb Z$. 
See (\ref{pushforward})-(\ref{defmu}).

On $%
\mathbb{Z}
^{r}\wr \Gamma_1(N),$ consider the measures $\phi _{i}$ supported on elements
of the form%
\[
g=(\delta ^{i},0)(0,\overline{s}_{i}^{m})(-\delta ^{i},0), 
\]%
where $\delta ^{i}:\mathbf{F}_r/N=\Gamma_1(N)\rightarrow 
\mathbb{Z}
^{r}$ is the function that's identically zero except that at identity $e$ of $%
\Gamma_1(N),$ $\delta ^{i}(e)=\epsilon_{i}\in \mathbb Z^r.$ 
For such $g,$ set (compare to (\ref%
{pushforward}))%
\[
\phi _{i}(g)=p_i(m). 
\]%
Note that 
\[
g^{-1}=(\delta ^{i},0)(0,\overline{s}_{i}^{-m})(-\delta ^{i},0) 
\]%
is an element of the same form, and $\phi _{i}(g^{-1})=\phi _{i}(g)=p_i(m).$ Set 
$$
\phi =\frac{1}{r}\sum_{i=1}^{r}\phi _{i}.  
$$
More formally, $\phi$ can be written as 
\begin{equation}
\forall\,g\in \mathbb Z^r\wr \Gamma_1(N),\;\;
\phi (g)=\sum_{1\leq i\leq r}\frac{1}{r}\sum_{m\in 
\mathbb{Z}
}p_{i}(m)\mathbf{1}_{\{(\delta ^{i},0)(0,\overline{s}%
_{i}^{m})(-\delta ^{i},0)\}}(g). \label{DefLower}
\end{equation}

Let $(U_{n})_1^\infty$ be a sequence of $\mathbf F_r$-valued  i.i.d.\ random
variables with distribution $\boldsymbol{\mu }$ and $
Z_{n}=U_{1}\cdots U_{n}.$ Note that the projection of $U_{n}$ to $\mathbf{F}_r
/[N,N]=\Gamma_2(N)$ (resp. $\mathbf{F}_r/N=\Gamma_1(N)$) is an i.i.d. sequence of $\Gamma_2(N)$-valued (resp. $\Gamma_1(N)$-valued) random
variables  with
distribution $\mu $ (resp. $\overline{\mu }$). Let $X_{i}$ denote the
projection of $U_{i}$ on  $\Gamma_1(N)$and  $T_{j}=X_{1}\cdots X_{j}.$
Consider the $\mathbb Z^r\wr \Gamma_1(N)$-valued random variable defined by 
\[
V_{n}= (\delta ^{i},0)(0,\overline{s}_{i}^{m})(-\delta ^{i},0) \mbox{ if }
U_{n}=\mathbf{s}_{i}^{m}. 
\]%
Then $(V_{n})_1^\infty$ is a sequence of i.i.d.\ random variables on $
\mathbb{Z}
^{r}\wr \Gamma_1(N)$ with distribution $\phi .$ Write 
\[
W_{n}=V_{1}...V_{n}. 
\]%
Then $W_{n}$ is the random walk on $\mathbb Z^r\wr \Gamma_1(N)$ driven by $\phi$.

The following proposition is based on Theorem \ref{th-flowdescr},
that is,\cite[Theorem 2.7%
]{Myasnikov2010}, which states 
that two words in $\mathbf{F}_r$ projects to the same
element in $\Gamma_2(N)$ if and only if they induce the same flow
on $\Gamma_1(N)$. In particular, the random walk on $\Gamma_2(N)$
returns to identity if and only if the path on $\Gamma_1(N)$ induces the
zero flow function.

\begin{pro}
\label{Lower} Fix a measure $\boldsymbol{\mu}$ on $\mathbf F_r$ of the form {\em (\ref{pushforward})}.
Suppose none of the $\overline{s}_{i}$ are torsion elements in $%
\Gamma_1(N)=\mathbf{F}_r/N.$ Let $\mu $ be  the probability measure on 
$\Gamma_2(N)$ defined at 
{\em (\ref{defmu})}. Let $\phi$ be the probability measure on $\mathbb Z^r\wr \Gamma_1(N)$ defined at   {\em (\ref{DefLower})}. It holds that 
\[
\mu ^{\ast n}(e_\ast)=\phi^{\star n}(e_\star). 
\]
\end{pro}

\begin{rem}
It's important here that the probability measure $\mu $ is supported 
on powers of generators, so that
each step is taken along one dimensional subgraphs $g\left\langle \overline{s%
}_{i}\right\rangle .$ The statement is not true for arbitrary measure on $%
\mathbf{F}/N^{\prime }$.
\end{rem}

\begin{proof}
The random walk $W_{n}$ on $%
\mathbb{Z}
^{r}\wr (\mathbf{F}/N)$ driven by $\phi $ can be written as 
\[
W_{n}=(f_{n},T_{n})=((f_{n}^{1},...,f_{n}^{r}),T_{n}). 
\]%
By definition of $W_n$, for any $x\in \Gamma_1(N)$, $f^i_n(x)$ counts the 
algebraic sums of the $i$-arrivals and $i$-departures of the random walk $T_n$  
at $x$ where by $i$-arrival (resp. $i$-departure) at $x$,
we mean a time $\ell$  at which $T_{\ell}=x$ and $U_{\ell}\in \langle 
\mathbf s_i \rangle$ (resp. $U_{\ell+1}\in \langle 
\mathbf s_i \rangle$). The condition $T_n=x\neq \bar{e}$ implies that the vector $f_n(x)$ 
must have at least one non-zero component because the total number of arrivals 
and departures at $x$ must be odd. Hence, we have
$$\phi^{\star n}(e_\star)=\mathbf P( (f_n,T_n)=e_\star)= 
\mathbf P(f^i_n(x)=0,\, 1\le i\le r,\, x\in \Gamma_1(N)).$$ 
We also have%
\[
\mu ^{\ast n}(e_\ast)=\mathbf P(\mathfrak{f}_{Z_{n}}(x,x\overline{s}_{i},\mathbf s_i)=0,
\,1\leq i\leq r,\,x\in \Gamma_1(N)). 
\]%
Given a flow $\mathfrak f$ on $\Gamma_1(N)$ (i.e., a function the edge set 
$\mathfrak E=\{(x,x\bar{s}_i,\mathbf s_i),x\in \Gamma_1(N),1\le i\le r\}\subset \Gamma_1(N)\times \Gamma_1(N)\times S$, for each $i, 1\le i\le r$, 
introduce the $i$-partial total flow $\partial_i \mathfrak f (x)$ at 
$x\in \Gamma_1(N)$ by setting 
$$\partial_i \mathfrak f (x)= \mathfrak 
f((x,x\bar{s}_i,\mathbf s_i))-\mathfrak f(x\bar{s}_i^{-1},x,\mathbf s_i).$$
It is easy to check (e.g., by induction on $n$) that
\begin{equation}\label{f=f}
\forall\,x\in \Gamma_1(N), \;\; f^i_n(x)=\partial_i\mathfrak f_{Z_n}(x).
\end{equation}
Obviously, $\mathfrak f_{Z_n}\equiv 0$ implies $f_n^i\equiv 0$ 
for all $1\le i\le r$ so 
$$\phi^{\star n}(e_\star)\ge \mu^{*n}(e_*).$$
But, in fact, under the assumption that none of the $\overline{s}_{i}$ 
are torsion elements in $\Gamma_1(N)$,
each edge $(x,x\overline{s}_{i},\mathbf s_i)$ 
in the Cayley graph of $\Gamma_1(N)$ is contained in the one
dimensional infinite linear subgraph $\{x\overline{s}_{i}^k: k\in \mathbb Z\}$
and, since $f_n$ and $\mathfrak f_{Z_n}$ are finitely supported, the equation
(\ref{f=f}) shows that $f_n^i\equiv 0$ implies that 
$\mathfrak f_{Z_n}(x,x\bar{s}_i,\mathbf s_i)=0$ for all $x\in \Gamma_1(N)$.
In particular, if $f_n^i\equiv 0$ for all $1\le i\le r$ then we must have 
$\mathfrak f_{Z_n}\equiv 0$.  Hence, if none of the $\bar{s}_i$ 
is a torsion element in $\Gamma_1(N)$, we have 
$f_n^i\equiv 0$, $1\le i\le r \Longleftrightarrow \mathfrak f_{Z_n}\equiv 0$
and thus $\mu^{*n}(e_{*})=\phi^{\star n}(e_\star)$. 
\if
On the other hand, recall that we use $f_{n}$ to denote lamp configuration
of $W_{n},$ then lamp configuration $f_{n}^{i}$ in coordinate $e_{i}$
records starting point and end point of each step taken in $\left\langle 
\overline{s}_{i}\right\rangle .$ Indeed, note that every time the random
walk moves along $g\cdot \left\langle \overline{s}_{i}\right\rangle ,$ the
lamp configuration move up $1$ in $i$-th coordinate at the starting point
(source) and move down $-1$ at the end point (sink). Therefore $%
\{f_{n}^{i}(x)=0$ for all $x\in g\cdot \left\langle \overline{s}%
_{i}\right\rangle \}$ is same as the event that the flow induced by $u_{n}$
along $g\cdot \left\langle \overline{s}_{i}\right\rangle $ is without
sources and sinks. Then by the Lemma \ref{circulation}, we have that for any 
$g\in \mathbf{F}/N,$ 
\begin{eqnarray*}
\{f_{n}^{i}(x) &=&0\text{ for all }x\in g\cdot \left\langle \overline{s}%
_{i}\right\rangle \} \\
&=&\{\text{restriction }\mathfrak{f}_{Z_{n}}\left\vert _{g\left\langle 
\overline{s}_{i}\right\rangle }\right. \text{ is a circulation}\} \\
&=&\{\mathfrak{f}_{Z_{n}}(x,x\overline{s}_{i})=0\text{ for all }x\in g\cdot
\left\langle \overline{s}_{i}\right\rangle \}.
\end{eqnarray*}%
Taking intersection over all edges $(g,g\overline{s}_{i}),$ we have 
\begin{eqnarray*}
\{f_{n}(x) &=&0\text{ for all }x\in \mathbf{F}/N\} \\
&=&\{\mathfrak{f}_{Z_{n}}(x,x\overline{s}_{i})=0\text{ for all }x\in \mathbf{%
F}/N,1\leq i\leq r\}.
\end{eqnarray*}%
Therefore 
\[
\phi ^{\star n}(\mathbf{e})=\mu ^{\ast n}(e). 
\]\fi
\end{proof}

In general, the probability measure $\phi$ on $\mathbb Z^r\wr \Gamma_1(N)$ 
does not have generating support because of the very specific and limited
nature of the lamp moves and how they correlate to the base moves. To fix this problem, let $\eta_r$ be the probability measure of the lazy random walk on 
$\mathbb Z^r$ so that $\eta_r(0)=1/2$ and 
$\eta_r(\pm \epsilon_i)=1/(4r)$, $1\le i\le r$. With this notation, let
\begin{equation}\label{defnusws}
q= \eta_r\star \bar{\mu}\star \eta_r
\end{equation}
be the probability measure of the switch-walk-switch random walk on the wreath product $\mathbb Z^r\wr \Gamma_1(N)$ associated with the walk-measure $\bar{\mu}$ on the base-group $\Gamma_1(N)$ and the switch-measure $\eta_r$ on the 
lamp-group $\mathbb Z^r$. See \cite{Pittet2002,SCZ-dv1} and Section \ref{sub-wr} 
for further details.

\begin{pro}
\label{LowerExpression}
Fix a measure $\boldsymbol{\mu}$ on $\mathbf F_r$ 
of the form {\em (\ref{pushforward})}.
Suppose that none of the $\overline{s}_{i}$ are torsion elements in $%
\Gamma_1(N)=\mathbf{F}_r/N.$ Refering to the notation introduced above,
there are $c,N\in (0,\infty)$ such that
the  probability measure  $\mu$ on $\Gamma_2(N)$ 
defined by {\em (\ref{defmu})} and the measure $q$ on 
$\mathbb Z^r\wr \Gamma_1(N)$  defined at {\em (\ref{defnusws})} satisfy
\[
\mu ^{\ast 2n}(e_*)\ge c q^{\star 2Nn}(e_\star). 
\]
\end{pro}

\begin{proof}
On a group $G$, the Dirichlet form associated with a symmetric measure $p$ is
 defined by $$\mathcal E_p(f,f)=\frac{1}{2}\sum_{x,y\in G}|f(xy)-f(x)|^2p(y).$$
From the definition, it easily follows that  $
\mathbb{Z}
^{r}\wr \Gamma_1(N),$ we have the  comparison of Dirichlet forms%
\[
\mathcal{E}_{\phi}\leq (2r)^{2}\mathcal{E}_{\eta _{r}\star \overline{%
\mu }\star \eta _{r}}= (2r)^{2}\mathcal{E}_{q} . 
\]%
Therefore, by \cite[Theorem 2.3]{Pittet2000}, 
\[
\phi^{\star 2n}(e_\star)\ge c q^{\star 2Nn}(e_\star). 
\]%
From Proposition \ref{Lower} we conclude that 
\[
\mu ^{\ast 2n}(e_*)=\phi^{\star 2n}(e_\star)\ge c q^{\star 2Nn}(e_\star). 
\]
\end{proof}

\begin{cor}\label{cor-low} Fix $a=(\alpha_1,\dots,\alpha_r)\in (0,2)^r$ and let
$\boldsymbol{\mu}_a$ be defined by {\em (\ref{pushforward})}
with $p_i(m)=c_i(1+|m|)^{-1-\alpha_i}$. Let $N=[\mathbf F_r,\mathbf F_r]$ so that
$\Gamma_1(N)=\mathbb Z^r$ and $\Gamma_2(N)=\mathbf S_{2,r}$. Let
$\mu_a$ be the probability measure on $\mathbf S_{2,r}$ associated 
to $\boldsymbol{\mu}_a$ by {\em (\ref{defmu})}.  Then we have
$$\mu_a^{*n}(e_*)\ge 
\exp\left( - C n^{r/(r+\alpha)}[\log n]^{\alpha/(r+\alpha)}\right)$$
where
$$\frac{1}{\alpha}=\frac{1}{r}\left(\frac{1}{\alpha_1}+\dots+\frac{1}{\alpha_r}\right).$$
\end{cor}
\begin{rem}Later we will prove a matching upper bound.
\end{rem}
\begin{proof} Proposition \ref{LowerExpression} yields 
$$\mu_a^{*n}(e_*)\ge 
c [\eta_r\star \bar{\mu}_a\star \eta_r]^{\star n}(e_\star)$$ 
where the probability $\bar{\mu}_a$ on $\Gamma_1(N)=\mathbb Z^r$ 
is defined at (\ref{defmu}) and is
given explicitely by 
$$\bar{\mu}_a(g)=\frac{1}{r}\sum_1^r p_i(m)\mathbf 1_{\bar{s_i}^n}(g)$$
where $\bar{s_i}$ canonical generators of $\mathbb Z^r$ and we have retained 
the multiplicative notation so that 
$\bar{s}_1^n=(n,0,\dots,0),\dots, \bar{s}_r^n=(0,\dots,0,n)$.

The behavior of the random walk on the wreath product 
$\mathbb Z^r\wr \mathbb Z^r$ associated with the switch-walk-switch measure 
$q=\eta_r\star \bar{\mu}_a\star \eta_r$ is studied in \cite{SCZ-dv1} 
where it is proved that
$$q^{\star 2n}(e_\star)\simeq 
\exp\left( -  n^{r/(r+\alpha)}[\log n]^{\alpha/(r+\alpha)}\right).$$
Corollary \ref{cor-low} follows. \end{proof}

\section{Return probability upper bounds}\setcounter{equation}{0}

This section explaines how to use the Magnus embedding 
(defined at (\ref{defpsi}))$$
\bar{\psi}: \mathbf F_r/[N,N]=\Gamma_2(N) \hookrightarrow 
\mathbb Z^r\wr \Gamma_1(N),\;\;\Gamma_1(N)=\mathbf F_r/N,$$
to produce, in certain cases, 
an upper bound on the probability of return $\mu^{*2n} (e_*)$ on 
$\Gamma_2(N)$.  Recall from (\ref{defbara}) that the Magnus embedding 
$\psi$ is described more concretely by
\begin{eqnarray*}
\Gamma_2(N) &\hookrightarrow &%
\mathbb{Z}
^{r}\wr \Gamma_1(N) \\
g &\mapsto & \bar{\psi}(g)=(\bar{a}(g),\overline{g}),\;\;\overline{g}=\bar{\pi}(g).
\end{eqnarray*} 
Here $\bar{a}(g)$ is an element of $\sum_{x\in \Gamma_1(N)} \mathbb Z^r_x$, 
equivalently, a $\mathbb Z^r$-valued function with finite support
defined on $\Gamma_1(N)$, equivalently, an element of the $\mathbb Z(\Gamma_1(N))$-module $\mathbb Z^r(\Gamma_1(N))$ . In any group $G$, we let $\tau_gx=gx$ be the 
translation by $g\in G$ on the left as well as its extension to any 
$\mathbb Z(G)$ module.  We will need the following lemma.

\begin{lem} \label{lem-alg-a}
For any $g,h\in \Gamma_2(N)$ 
with $\bar{g}=\bar{\pi}(g)\in \Gamma_1(N)$, we have
$$\bar{a}(gh)= \bar{a}(g)+ \tau_{\bar{g}}\bar{a}(h).$$
In particular, if $g\in \Gamma_2(N)$ and $\boldsymbol \rho\in N$ with 
$\rho=\pi_2(\boldsymbol \rho)\in \Gamma_2(N)$, 
we have
$$\bar{a}(g\rho g^{-1})=\tau_{\bar{g}} \bar{a}(\rho).$$
\end{lem}
\begin{proof} The first formula follows from the Magnus embedding by inspection.
The second formula is an easy consequence of the first and the fact that 
$\pi(\boldsymbol \rho)$ is the identity element in $\Gamma_1(N)$.
\end{proof} 
\begin{rem}\label{rem-alg-a} The identities stated in Lemma \ref{lem-alg-a}
can be equivalently written in terms of flows on $\Gamma_1(N)$. Namely, 
for $\mathbf u,\mathbf v \in \mathbf F_r$, we have
$$\mathfrak f_{\mathbf u \mathbf v}= \mathfrak f_{\mathbf u} +
\tau_{\pi(\mathbf u)}\mathfrak f_{\mathbf v}
\mbox{ and }\;\;
\mathfrak f_{\mathbf u \mathbf v \mathbf u^{-1}}=
\tau_{\pi(\mathbf u)}\mathfrak f_{\mathbf v}.$$ 
\end{rem}

\subsection{Exclusive pairs} \label{sec-exclu}

\begin{defin}\label{def-exclu}
Let $\Gamma$ be a subgroup of $\Gamma_2(N)$ and $\boldsymbol \rho$
be a reduced word in $N\setminus [N,N] \subset \mathbf F_r$.  Let $\overline{\Gamma}=\bar{\pi}(\Gamma)$. Set $\rho=\pi_2(\boldsymbol \rho)\in \Gamma_2(N)$. 
We say the 
pair $(\Gamma,\boldsymbol \rho)$ is {\em exclusive} if the following 
two conditions are satisfied:
\begin{description}
\item[(i)] The collection $\{\tau _{\overline{g}}(\bar{a}(\rho ))\}_{\overline{g}%
\in \overline{\Gamma }}$ is $%
\mathbb{Z}
$-independent in the $\mathbb Z$-module $\sum_{\Gamma_1(N)}(
\mathbb{Z}
^{r})_{x}$.

\item[(ii)] In the $\mathbb Z$-module $\sum_{\Gamma_1(N)}(%
\mathbb{Z}
^{r})_{x},$ the $\mathbb Z$-submodule generated by 
$\{\tau _{\overline{g}}(\bar{a}(\rho
))\}_{\overline{g}\in \overline{\Gamma }}$, call it $A=A(\Gamma,\boldsymbol \rho)$, has trivial intersection with
the subset 
$B =\{\bar{a}(g):g\in \Gamma \}$, that is 
\[
A\cap B=\{\mathbf{0}\}. 
\]
\end{description}
\end{defin}
\begin{rem}\label{rem-wr-rho} Condition (i) implies that the 
$\mathbb Z$-submodule $A(\Gamma,\boldsymbol \rho)$ of $\sum_{\Gamma_1(N)}(\mathbb Z^r)_x$
is isomorphic to $\sum _{\bar{g}\in
\overline{\Gamma}} (\mathbb Z)_{\bar{g}}$.
\end{rem}

\begin{exa}\label{meta-exclu}
In the free metabelian group $S_{2,r}=\mathbf{F}_r/[N,N]$, 
$N=[\mathbf F_r,\mathbf F_r]$, set
$\Gamma =\left\langle
s_{1}^{2},...,s_{r}^{2}\right\rangle,$ and 
$\boldsymbol{\rho }=[\mathbf{s}_{1},\mathbf{s}%
_{2}].$ Then $(\Gamma ,\boldsymbol{\rho })$ is an exclusive pair.
The conditions (i)--(ii) are easy to check because the monomials $%
\{Z_{1}^{x_{1}}Z_{2}^{x_{2}}...Z_{r}^{x_{r}}:x\in 
\mathbb{Z}
^{d}\}$ are $%
\mathbb{Z}
$-linear independent in $%
\mathbb{Z}
(%
\mathbb{Z}
^{r}).$ A similar idea was used in the proof of \cite[Theorem 3.2]{Erschler2004}.
\end{exa}

We now formulate a sufficient condition for  
a pair $(\Gamma ,\boldsymbol{\rho })$ to be exclusive. This sufficient condition
is phrased in terms of the representation of the elements of $\Gamma_2(N)$ 
as flows on $\Gamma_1(N)$. Recall that $\Gamma_1(N)$ come equipped 
with a marked Cayley graph structure as described in Section \ref{sec-flow}.


\begin{lem}\label{lem-exclu}
Fix $\Gamma <\Gamma_2(N)$ and $\boldsymbol{\rho}$ as in 
{\em Definition \ref{def-exclu}}. Set 
$$
U=\bigcup _{g\in \Gamma }\text{\em supp} (\mathfrak{f}_{g}), 
$$
that is the union of the support of the flows on $\Gamma_1(N)$ induced by
elements of $\Gamma $. Assume that 
 $\boldsymbol{\rho}=\mathbf{us}\mathbf{v}$ with $\mathbf s\in \{\mathbf s_1,\dots,\mathbf s_r\}$ and  that:
\begin{enumerate}
\item $\mathfrak{f}_{\boldsymbol{\rho }}((\overline{u},\overline{u}\overline{s}%
, \mathbf s))\neq 0$;
\item For all $x\in \overline{\Gamma }\backslash \{\overline{e}\},$ $%
\mathfrak{f}_{\boldsymbol{\rho }}((x\overline{u},x\overline{u}\overline{s},
\mathbf s))=0$;
\item The  edge $(\overline{u},\overline{u}%
\overline{s},\mathbf s))$ is not in $ U$.
\end{enumerate}
Then the pair $(\Gamma ,\boldsymbol{\rho})$ is exclusive.
\end{lem}
\begin{rem} \label{rem-exclu}
The first assumption insures that the given edge is really active in 
the loop associated with $\mathbf \rho$ on the Cayley graph $\Gamma_1(N)$.
The proof given below shows that conditions 1-2 above imply condition (i) of Definition \ref{def-exclu}. 
All three assumptions above are used to obtain condition (ii)
of Definition \ref{def-exclu}. 
\end{rem}
\begin{proof}
Condition (i). Suppose there is a nontrivial linear relation such that 
\[
c_{1}\tau _{\overline{g}_{1}}(\bar{a}(\rho ))+...+c_{n}\tau _{\overline{g}%
_{n}}(\bar{a}(\rho ))=0, \;\;c_i\in \mathbb Z,\]%
where some  $c_{j}$, say $c_1$, is not zero and
the element $\overline{g}_{j}\in \overline{\Gamma}$ are  pairwise distinct. 
Let $\mathbf g_j$ be representative of $\bar{g}_j$ in $\mathbf F_r$.
Let $b$ denote the
coefficient of $\sum_{i=1}^{n}c_{i}\tau _{\overline{g}_{i}}(\bar{a}(\rho ))$ in
front of the term $\overline{g}_{1}\overline{u}\lambda _{\mathbf{s}}$. 
By formula (\ref{FlowCoefficient}),%
\[
b=\sum_{i=1}^{n}c_{i}\mathfrak{f}_{\mathbf g_{i}\boldsymbol \rho \mathbf 
g_{i}^{-1}}((\overline{g}_{1}%
\overline{u},\overline{g}_{1}\overline{u}\overline{s},\mathbf s)). 
\]%
Note that%
\[
\mathfrak{f}_{\mathbf g_{i}\boldsymbol \rho \mathbf g_{i}^{-1}}(
(\overline{g}_{1}\overline{u},\overline{g}
_{1}\overline{u}\overline{s},\mathbf s))=\mathfrak{f}_{\boldsymbol
\rho }((\overline{g}%
_{i}^{-1}\overline{g}_{1}\overline{u},\overline{g}_{i}^{-1}\overline{g}_{1}%
\overline{u}\overline{s},\mathbf s)). 
\]%
Therefore, 
since $\overline{g}_{i}^{-1}\overline{g}_{1}\neq 
\overline{e}$ for all $i\neq 1,$ the hypothesis stated in 
Lemma \ref{lem-exclu}(2) 
gives
 $$\forall\, i\neq 1,\;\;\mathfrak{f}_{\mathbf g_{i}\boldsymbol \rho \mathbf g_{i}^{-1}}
((\overline{g}_{1}%
\overline{u},\overline{g}_{1}\overline{u}\overline{s}, \mathbf s))=0.$$ By 
hypothesis (1) of  Lemma \ref{lem-exclu}, this implies 
\[
b=c_{1}\mathfrak{f}_{\boldsymbol{\rho }}((\overline{u},\overline{u}\overline{s},
\mathbf s))\neq 0
\]%
which provides a contradiction. 

We now verify that  Condition  (ii) of Definition \ref{def-exclu} holds.
Fix $x\in A\cap B$ and assume that  $x$ is nontrivial. From Condition  (i), 
$x$ can be written uniquely as 
$$x=c_{1}\tau _{\overline{g}_{1}}a(\rho
)+...+c_{n}\tau _{\overline{g}_{n}}a(\rho ),$$ where  $c_{j}\in 
\mathbb{Z}
\backslash \{0\}$ and the elements  $\overline{g}_{j}$ are pairwise distinct. 
On the other hand, since $%
x\in B,$ there exists some $h\in \Gamma $ such that $x=\bar{a}(h)$. 
By formula (\ref{FlowCoefficient}), 
$\bar{a}(h)=\sum_{i=1}^{n}c_{i}\tau _{\overline{g}_{i}}a(\rho
)$ is equivalent to 
\[
\mathfrak{f}_{h}=\sum_{i=1}^{n}c_{i}\mathfrak{f}_{g_{i}\rho g_{i}^{-1}}.
\]%
Therefore $\mathfrak{f}_{g_{1}^{-1}hg_{1}}=\sum_{i=1}^{n}c_{i}\mathfrak{f}%
_{g_{1}^{-1}g_{i}\rho g_{i}^{-1}g_{1}}$. By hypothesis (2), it follows that
$$\mathfrak{f}%
_{g_{1}^{-1}hg_{1}}((\overline{u},\overline{u}\overline{s},\mathbf s))
=c_{1}\mathfrak f _{\rho}((\overline{u},\overline{u}\overline{s},\mathbf s))\neq 0.$$
However, since $g_{1}^{-1}hg_{1}\in \Gamma $, this implies that
$(\overline{u},\overline{u}\overline{s},\mathbf s_i))\in  U$, a
conclusion which contradicts  assumption (3). Hence $A\cap B=\{\mathbf 0\}$ as desired.

\end{proof}

\subsection{Existence of exclusive pairs}\label{sec-exexclu}
This section discuss algebraic conditions that allow us to produce 
approriate exclusive pairs.

\begin{lem}\label{lem-resf}
Assume $\Gamma_1(N)=\mathbf{F}_{r}/N$ is residually finite and 
$r\geq 2$. Fix an
element $\boldsymbol{\rho }$ in $N\setminus [N,N]$. 
There exists a finite
index normal subgroup $K=K_{\boldsymbol \rho}\vartriangleleft \Gamma_1(N)$ 
such that, for any  edge $(\mathbf{u},\mathbf{us}%
)$ in $\boldsymbol{\rho}=\mathbf{us}\mathbf{v}$ 
with $\mathbf s\in \{\mathbf s_1,\dots,\mathbf s_r\}$ and  
any subgroup $H< \Gamma_2(N)$ with $\overline{\pi }(H)<K$,
$$\forall \,x\in \bar{\pi}(H)
\setminus \{\overline{e}\},\;\;
\mathfrak{f}_{\boldsymbol{\rho }}((x\overline{u},x\overline{u}\overline{s},
\mathbf s))=0.$$
\end{lem}
\begin{rem}
Since $\boldsymbol{\rho } \notin [N,N]$, the flow induced by $\boldsymbol{
\rho }$ is not identically zero. Therefore, after changing $\boldsymbol \rho$ 
to $\boldsymbol \rho^{-1}$ 
if necessary, there exists a reduced word 
$\mathbf u$ and $i\in\{1,\dots, r\}$ such that $\boldsymbol \rho= \mathbf u \mathbf s_i \mathbf u'$
and $
\mathfrak{f}_{\boldsymbol{\rho }}
((\overline{u},\overline{u}\overline{s}_{i},\mathbf s_i))\neq 0$. Hence
Lemma \ref{lem-resf} provides a way to verify 
conditions 1 and 2 of Lemma \ref{lem-exclu}.
\end{rem}

\begin{proof}
For any element $\boldsymbol{\rho }$ in $N\setminus [N,N]$, view $\boldsymbol{%
\rho }$ as a reduced word in $\mathbf F_r$. 
Let $B_{\boldsymbol{\rho }}$ be  the collection of
all proper subword $\mathbf{u}$ of $\boldsymbol{\rho }$ such that $\bar{\pi} 
(\mathbf{u})$ is not trivial in $\Gamma_1(N)$. 
Since $\Gamma_1(N)$ is residually
finite, there exists a normal subgroup $K\vartriangleleft \Gamma_1(N)$ 
such that $\Gamma_1(N)/K$ is finite and 
$\bar{\pi }(B_{\mathbf{\rho }})\cap $ $K=\emptyset $.

Suppose there exists $x\in \bar{\pi}(H)$ such that $x$ is not trivial  
and $\mathfrak{f}_{\mathbf{\rho }}((x\overline{u},x\overline{u}%
\overline{s},\mathbf s))\neq 0$. 
Therefore, there is a proper subword 
$\mathbf{v}$ of $\boldsymbol{\rho }$ such that 
$\boldsymbol{\rho }=\mathbf{vw}$ and $
\overline{v}=x\overline{u}$.  
Since both $\mathbf u$ and $\mathbf v$ are 
prefixes of $\boldsymbol \rho$ and $x$ is not trivial, 
$\mathbf v\mathbf u^{-1}$ is the conjugate of a proper subword of 
$\boldsymbol \rho$ with non-trivial image in $\Gamma_1(N)$.
By construction this implies that $\bar{\pi} (\mathbf{vu}^{-1})\notin K$,
a contradition since $\bar{\pi} (\mathbf{vu}^{-1})=x\in \bar{\pi}(H)<K$.
\end{proof}

\begin{rem} By a classical result of Hirsch, polycyclic groups 
are residually finite, \cite[5.4.17]{Rob}. 
By a result of P. Hall, a finitely generated group 
which is an extension of an abelian group by a nilpotent group is residually 
finite.  In particular, all finitely generated metabelian groups
are residually finite, \cite[15.4.1]{Rob}. Gruenberg, \cite{Grue}, proves that
free polynilpotent groups are residually finite. The free solvable groups $\mathbf S_{d,r}$ are examples of free polynilpotent groups. 
Note that all finitely generated residually finite groups are Hopfian, 
\cite[6.1.11]{Rob}.
\end{rem}

Our next task is to find ways to verify condition 3 of Lemma \ref{lem-exclu}.
To this end, 
let $A$ be the abelian group $\Gamma_1(N)/[\Gamma_1(N),\Gamma_1(N)]$. 
Fix $m=(m_1,\dots,m_r)\in \mathbb  N^r$ and let $A_m$ be the subgroup of $A$ generated by the
images of the elements $\bar{s}_i^{m_i}$, $1\le i\le r$, in $A$. Let
$T_m$ be the finite Abelian group $T_m=A/A_m$. 
Let $\pi_{T_m} :\mathbf F_r \ra T_m$
be the projection from $\mathbf F_r$ onto $T_m$. Set
$$H_m=
\left\langle s_{i}^{m_i},1\leq i\leq r\right\rangle <\Gamma_2(N).$$

\begin{lem} \label{lem-Tm}
Fix a reduced word  $\boldsymbol{\rho }\in N\backslash [N,N]$. 
Assume that 
$\boldsymbol{\rho }=\mathbf{us}\mathbf{v}$, 
where $\mathbf{s}\in \{\mathbf s_1,\dots,\mathbf s_r\}$
and $\mathfrak{f}_{\mathbf{\rho }}((\overline{%
u},\overline{u}\overline{s},\mathbf s))\neq 0$. Fix $m\in \mathbb N^r$ 
and assume 
that, in the 
finite abelian group $T_m$,  $\pi_{T_m}(\mathbf{u})\notin
\left\langle \pi_{T_m }(\mathbf{s})\right\rangle $. Then
the  edge $(\overline{u},\overline{u}%
\overline{s},\mathbf s))$ is not in 
$$U(H_m) =\bigcup _{g\in H_m }\text{\em supp} (\mathfrak{f}_{g}).  $$
\end{lem}
\begin{proof} Assume that 
$(\overline{u},\overline{u}\overline{s},\mathbf s)\in U(H_m)$. Then there
must exist $x\in \pi (H_m )$ and $q\in 
\mathbb{Z}
$ such that 
$
x\bar{s}^{q}=\overline{u}.
$
But, projecting on $T_m$, 
this contradicts the assumption $\pi_{T_m }(\mathbf{u})\notin
\left\langle \pi_{T_m }(\mathbf{s})\right\rangle $.
\end{proof}

We now put together these two lemmas and state a proposition that will 
allow us to produce exclusive pairs.
\begin{pro} \label{pro-resf}
Fix $N\vartriangleleft \mathbf F_r$ and $\boldsymbol \rho
\in N\setminus [N,N]$, in reduced form. Let $\mathbf u$ be a prefix of 
$\boldsymbol \rho$ such that $\boldsymbol \rho= \mathbf u\mathbf s\mathbf v$,
$\mathbf s\in \{\mathbf s_1,\dots,\mathbf s_r\}$
and $\mathfrak f _{\boldsymbol \rho}((\bar{u},\bar{u}\bar{s},\mathbf s))\neq 0$.
Assume that the group $\Gamma_1(N)$ is residually finite and 
there exists an integer vector  
$m=(m_1,\dots,m_r)\in \mathbb N^r$ such that, in the finite group 
$T_m$, $\pi_{T_m}(\mathbf u)\notin \langle 
\pi_{T_m}(\mathbf s)\rangle$.  Then there is $m'=(m'_1,\dots,m'_r)$ such that
the pair $(H_{m'},\boldsymbol \rho)$ is an exclusive pair in $\Gamma_2(N)$.
\end{pro}
\begin{proof} Let $K_{\boldsymbol \rho}$ the the finite index normal subgroup 
of $\Gamma_1(N)$ given by Lemma \ref{lem-resf}. 
Since $K_{\boldsymbol \rho} $ is of finite index in $\Gamma_1(N)$, we can pick
$m'_i$ to be a multiple of $m_i$ such that $\bar{s}_i^{m'_i}\in K_{\boldsymbol 
\rho}$. Observe that the assumption 
$\pi_{T_m}(\mathbf u)\notin \langle 
\pi_{T_m}(\mathbf s)\rangle$ implies the same property with $m$ replaced by $m'$.
Applying Lemmas \ref{lem-resf}, \ref{lem-Tm} and Lemma \ref{lem-exclu} 
yields that $(H_{m'},\boldsymbol \rho)$ is an exclusive pair in $\Gamma_2(N)$.
\end{proof}

We conclude this section with a concrete  application of Proposition 
\ref{pro-resf}.

\begin{pro}\label{pro-exclunil}
Assume that $\Gamma _{1}(N)=\mathbf F_r/N$ 
is an infinite nilpotent group and $r\ge 2$. Then there
exists an exclusive pair $(\Gamma ,\mathbf{\rho })$ in $\Gamma _{2}(N)$
such that $\pi (\Gamma )$ is a subgroup of finite index in $\Gamma
_{1}(N)$.
\end{pro}

\begin{proof} First we construct an exclusive pair using Proposition 
\ref{pro-resf}. Suppose that $\Gamma _{1}(N)$ is not virtually $\mathbb{Z}$. 
Then the torsion-free rank of $
\Gamma _{1}(N)/[\Gamma _{1}(N),\Gamma _{1}(N)]$ is at least  $2$. Choose two
generators $s_{i_{1}},s_{i_{2}}$ such that their projections in the
abelianization are $\mathbb{Z}$-independent. Choose $\boldsymbol {\rho }$ to be
an element of minimal length in $N\cap \left\langle \mathbf{s}_{i_{1}},%
\mathbf{s}_{i_{2}}\right\rangle $. Note that since $\Gamma _{1}(N)$ is
nilpotent, this intersection contains commutators of $\mathbf{s}_{i_{1}},%
\mathbf{s}_{i_{2}}$ with length greater than the nilpotency class, 
therefore it is
non-empty. Proposition \ref{pro-resf} applies and yields an integer $m$
such that 
$(\Gamma =\left\langle
s_{1}^{m},...,s_{r}^{m}\right\rangle,\boldsymbol{\rho })$ is an exclusive pair.

In the special case when $\Gamma _{1}(N)$ is virtually $\mathbb{Z}$, choose 
$\boldsymbol{\rho }$ to be an element of minimal length in $N$, 
and a generator $%
\mathbf s_{i_{1}}$ such that $\bar{s}_{i_1}$ 
is not a torsion element in $\Gamma_1(N)$. Set $\Gamma
=\left\langle s_{i_1}^{m}\right\rangle $ with 
$m=[\Gamma _{1}(N):K_{\boldsymbol{%
\rho }}]$. Then by Lemmas \ref{lem-exclu}, Lemma \ref{lem-resf}  
and inspection,  $(\Gamma ,\boldsymbol{\rho })$ is an exclusive pair.

Next we use induction on nilpotency class $c$ to show 
that, for any $m\in \mathbb{N}$, $\pi (\Gamma )=\left\langle 
\overline{s}_{1}^{m},...,\overline{s}_{r}^{m}\right\rangle $ is a subgroup
of finite index in $\Gamma _{1}(N)$.
When $c=1$, observe that the statement is
obviously true for finitely generated abelian groups. Suppose $\Gamma _{1}(N)
$ is of nilpotency class $c$.  Let $H=\gamma_c(\Gamma_1(N))$. 
Using the induction 
hypothesis,  it suffices to prove that $H\cap \pi
(\Gamma )$ is a finite index subgroup of $H$. Note that $%
H$ is contained in the center of $\Gamma_11(N)$ and 
is generated by commutators of length $c$. Further,
\[
\lbrack
s_{i_{c}}[...[s_{i_{2}},s_{i_{1}}]]]^{m^{c}}=[s_{i_{c}}^{m}[...[s_{i_{2}}^{m},s_{i_{1}}^{m}]]].
\]%
Therefore $H/H\cap \pi (\Gamma )$ is a
finitely generated torsion abelian group, hence finite, as desired.
\end{proof}

\subsection{Random walks associated with exclusive pairs}\label{sec-up} 

The following result captures the main idea and construction of this section.
\begin{theo} \label{theo-Upper}
Let $\mu$ be a symmetric probability measure on $\Gamma_2(N)$.
Let $\Gamma <\Gamma_2(N)$ and $\boldsymbol{\rho} $ 
be an exclusive pair as in {\em Definition \ref{def-exclu}}. 
Set $\rho=\pi_2(\boldsymbol\rho)\in \Gamma_2(N)$. Let $\nu$ be the probability measure on $\Gamma_2(N)$ such that
$$\nu(\rho^{\pm 1})=1/2.$$
Let $\varphi$ be a symmetric probability measure on $\Gamma$ such that 
\begin{equation}\label{comp-exclu}
\mathcal E _{\nu*\varphi*\nu}\le C_0 \mathcal E_\mu.
\end{equation}
Let $\bar{\varphi}$ be the symmetric probability on 
$\overline{\Gamma}=\bar{\pi}(\Gamma)< \Gamma_1(N)$ defined by 
$$\forall\,\bar{g}\in \Gamma_1(N),\;\;\bar{\varphi}(\bar{g})=\varphi(\bar{\pi}^{-1}(\bar{g})).$$ 
On the wreath product 
$\mathbb Z\wr \overline{\Gamma}$ (whose group law will be denoted here  by $\star$),
consider the switch-walk-switch measure $q=\eta\star \bar{\varphi}\star \eta$
with $\eta(\pm 1)=1/2$ on $\mathbb Z$. Then there are constants $C,k\in (0,\infty)$  
such that
$$\mu^{*2kn}(e_*)\le C q^{\star 2n}(e_\star).$$ 
\end{theo} 
\begin{proof} By \cite[Theorem 2.3]{Pittet2000}, the comparison assumption 
between the Dirichlet forms of $\mu$ and $\nu*\varphi*\nu$ implies that
there is a constant $C$ and an integer $k$ such that
$$\forall \,n,\;\;\mu^{*kn}(e_*)\le C [\nu*\varphi*\nu]^{*2n}(e_*).$$
Hence, the desired conclusion easily follows from the next proposition.
\end{proof}

\begin{pro}
\label{Upper}
Let $\Gamma <\Gamma_2(N)$ and $\boldsymbol{\rho}$ be an exclusive pair as in 
{\em Definition \ref{def-exclu}}.
Let $\rho=\pi(\boldsymbol{\rho})$ and let $\nu$ be the probability measure 
on $\Gamma_2(N)$ such that $\nu (\rho )=\nu (\rho ^{-1})=\frac{1}{2}$. 
Let $\varphi$ be a probability measure supported on $\Gamma
$. Let  $\bar{\varphi }$ be the pushforward of $\varphi$ on $\bar{\pi}(\Gamma)=
\overline{\Gamma}< \Gamma_1(N)$. Let $\eta$ be the probability measure 
on $\mathbb Z$ such that $\eta(\pm 1)=1/2$.
Let $q=\eta \star 
\bar{\varphi }\star \eta $ be the switch-walk-switch measure
on $ \mathbb{Z}
\wr \overline{\Gamma}$. Then%
\[
\left( \nu \ast \varphi \ast \nu \right) ^{\ast n}(e_\ast)\leq (\eta \star 
\overline{\varphi }^{\prime }\star \eta )^{\star n}(e_\star)
=q^{\star n}(e_\star). 
\]
\end{pro}
To prove this proposition, we will use the following lemma.

\begin{lem}
\label{Return Expression}
Let $\varphi$ be a probability measure on $\Gamma_2(N)$. Let $\nu$
be the uniform  measure  on $\{r_0^{\pm 1}\}$ where 
$r_0\in \Gamma_2$ and $r_0\neq r_0^{-1}$. Let $(Y_i)_1^\infty$ and 
$(\varepsilon_i)_1^\infty$ be i.i.d.\ sequence with law $\varphi$ and $\nu$
respectively. Let $S_n=Y_1\cdots Y_n$ and 
$\overline {S}_n=\bar{\pi}(S_n)$. Then we have
 \begin{eqnarray*}
\lefteqn{\left( \nu \ast \varphi \ast \nu \right) ^{\ast n}(e_*)}\hspace{.5in} && \\
&=&
\mathbf P\left( \overline{S}_n=\bar{e},\sum_{j=1}^{n}(\varepsilon _{2j-1}+\varepsilon
_{2j})\tau _{\overline{S}_{j}}\bar{a}(r_{0})+\bar{a}(S_{n})=0\right) .
\end{eqnarray*}
\end{lem}
\begin{proof}
The product $
r_{0}^{\varepsilon _{1}}Y_{1}r_{0}^{\varepsilon _{2}+\varepsilon
_{3}}Y_{2}r_0^{\varepsilon_4}\cdots r_0^{\varepsilon _{2n-1}}Y_{n}r_{0}^{\varepsilon _{2n}}
$ has distribution $$(\nu \ast
\varphi \ast \nu )^{\ast n}. 
$$
Therefore, we have
\begin{eqnarray*}
(\nu \ast \mu \ast \nu )^{\ast n}(e_*) &=&\mathbf P(r_{0}^{\varepsilon
_{1}}Y_{1}r_{0}^{\varepsilon _{2}+
\varepsilon _{3}}Y_{2}\cdots Y_{n}r_{0}^{\varepsilon
_{n}}=e_*) \\
&=&\mathbf P(Y_{1}r_{0}^{\varepsilon _{2}+\varepsilon _{3}}Y_{2}\cdots
Y_{n}r_{0}^{\varepsilon _{2n}+\varepsilon _{1}}=e_*).
\end{eqnarray*}%
Using the Magnus embedding 
\[
\psi :\mathbf{F}/[N,N]\hookrightarrow \mathbb Z^r\wr \Gamma_1(N)
\]%
(and re-indexing of the $\varepsilon_i$) this yields
$$(\nu \ast \mu \ast \nu )^{\ast n}(e_*) \\
=\mathbf P(\overline{S}_{n}=\bar{e},\bar{a}(Y_{1}r_{0}^{\varepsilon _{1}+\varepsilon
_{2}}Y_{2}\cdots Y_{n}r_{0}^{\varepsilon _{2n-1}+\varepsilon _{2n}})=0).
$$
However, we have
\begin{eqnarray*}
\lefteqn{\bar{a}(Y_{1}r_{0}^{\varepsilon _{1}+
\varepsilon _{2}}Y_{2}\cdots Y_{n}r_{0}^{\varepsilon
_{2n-1}+\varepsilon _{2n}})} \hspace{1in}&& \\
&=&\bar{a}(S_{1}r_{0}^{\varepsilon _{1}+\varepsilon _{2}}S_{1}^{-1}\cdots
S_{n}r_{0}^{\varepsilon _{2n-1}+\varepsilon _{2n}}S_{n}^{-1} S_{n}) \\
&=&\sum_{j=1}^{n}(\varepsilon _{2j-1}+\varepsilon
_{2j})\bar{a}(S_{j}r_{0}S_{j}^{-1})+\bar{a}(S_{n}) \\
&=&\sum_{j=1}^{n}(\varepsilon _{2j-1}+\varepsilon _{2j})\tau
_{\overline{S}_{j}}\bar{a}(r_{0})+\bar{a}(S_{n}).
\end{eqnarray*}%
The last equality above 
from Lemma \ref{lem-alg-a}.
\end{proof}

\begin{proof}[Proof of Proposition \ref{Upper}]
By Lemma \ref{Return Expression}, 
\begin{eqnarray*}
\lefteqn{\left( \nu \ast \mu \ast \nu \right) ^{\ast n}(e_*)}\hspace{.5in}&& \\
&=&\mathbf P\left( \overline{S}_{n}=\bar{e},
\sum_{j=1}^{n}(\varepsilon _{2j-1}+\varepsilon
_{2j})\tau _{\overline{S}_{j}}\bar{a}(r_{0}))+\bar{a}(S_{n})=0\right) .
\end{eqnarray*}%
Under the assumption that $(\Gamma ,\mathbf{\rho })$ is an exclusive pair,
(ii) of Definition \ref{def-exclu} gives
\begin{eqnarray}
\lefteqn{\left\{ \sum_{j=1}^{n}(\varepsilon _{2j-1}+\varepsilon _{2j})\tau
_{\overline{S}_{j}}\bar{a}(r_{0})+\bar{a}(S_{n})=0\right\}} \hspace{.5in}&&
\nonumber\\
&&=\left\{ \sum_{j=1}^{n}(\varepsilon _{2j-1}+\varepsilon _{2j})\tau
_{\overline{S}_{j}}\bar{a}(r_{0})=0\right\}  
\cap \left\{ \bar{a}(S_{n})=0\right\}  \label{drop}  
\end{eqnarray}
Further, (i) of Definition \ref{def-exclu} gives
\begin{eqnarray*}
\lefteqn{\left\{ \sum_{j=1}^{n}(\varepsilon _{2j-1}+\varepsilon _{2j})\tau
_{\overline{S}_{j}}\bar{a}(r_{0})=0\right\}}\hspace{1in}&&\\ 
&= &\bigcap\limits_{x\in \overline{\Gamma }}\left\{
\sum_{j=1}^{n}(\varepsilon _{2j-1}+\varepsilon _{2j})
\mathbf{1}_{\{x\}}(\overline{S}_j)=0%
\right\} .
\end{eqnarray*}%
Therefore, dropping $\left\{ \bar{a}(S_{n})=0\right\} $ in (\ref{drop}) yields 
\begin{eqnarray*}
\lefteqn{\left( \nu \ast \varphi \ast \nu \right) ^{\ast n}(e_*)}
\hspace{.5in} \\
&\leq &
\mathbf P\left( \overline{S}_{n}=\bar{e},
\sum_{j=1}^{n}(\varepsilon _{2j-1}+\varepsilon
_{2j})\mathbf{1}_{\{x\}}(\overline{S}_j)=0\text{ for all }x\in \overline{\Gamma }%
\right) .
\end{eqnarray*}%
On the other hand, the return probability of the random walk on 
$$
\mathbb{Z}\wr \overline{\Gamma}<
\mathbb{Z}\wr \Gamma_1(N)$$ driven by $\eta \star \overline{\varphi }^{\prime }\star \eta $
is exactly%
\begin{eqnarray*}
\lefteqn{(\eta \star \overline{\varphi }\star \eta )^{\star n}(e_\star)
} \hspace{.5in}&& \\
&=&\mathbf P\left( \overline{S}_{n}=\bar{e},\sum_{j=1}^{n}(\varepsilon _{2j-1}+\varepsilon _{2j})%
\mathbf{1}_{\{x\}}(\overline{S}_j)=0\text{ for all }x\in \overline{\Gamma }\right) .
\end{eqnarray*}
\end{proof}

\section{Examples of two sided bounds on $\Phi_{\Gamma_2(N)}$}
\setcounter{equation}{0}

\subsection{The case of nilpotent groups} \label{sec-nil}
Our first application of the techniques developed above yields
the following Theorem.

\begin{theo} \label{theo-nil}
Assume that $\Gamma_1(N)=\mathbf F_r/N$ is an infinite nilpotent
group and $r\ge 2$. Let $D$ be the degree of polynomial 
volume growth of $\Gamma_1(N)$. Then
$$\Phi_{\Gamma_2(N)}(n)\simeq \exp\left(-n^{D/(2+D)}[\log n]^{2/(2+D)}\right).
$$
\end{theo}
\begin{proof}Example \ref{exa-pol} provides the desired lower bound.
By Proposition \ref{pro-exclunil}, we have an exclusive pair 
$(\Gamma,\boldsymbol \rho)$ in $\Gamma_2(N)$ such that 
$\overline{\Gamma}=\bar{\pi}(\Gamma)$ is of finite index in $\Gamma_1(N)$. 
Applying Theorem \ref{theo-Upper} gives
$$\Phi_{\Gamma_2(N)}(kn) \le C \Phi_{\mathbb Z\wr \overline{\Gamma}}(n).$$
 Since $\overline{\Gamma}$ has finite index in $\Gamma_1(N)$, it has the same 
volume growth degree $D$ and, by \cite[Theorem 2]{Erschler2006},
$$\Phi_{\mathbb Z\wr \overline{\Gamma}}(n)\le 
\exp\left(- cn^{D/(2+D)}[\log n]^{2/(2+D)}\right).$$
\end{proof}

\subsection{Application to the free metabelian groups}\label{sec-meta}

This section is devoted to the 
free metabelian group $\mathbf S_{2,r}
=\mathbf{F}/[N,N]$, 
$N=[\mathbf F_r,\mathbf F_r]$. 
\begin{theo} \label{theo-meta}
The free metabelian group $\mathbf S_{2,r}$ satisfies
\begin{equation}\label{meta1}
\Phi_{\mathbf S_{2,r}}(n)\simeq 
\exp\left(- n^{r/(2+r)}[\log n]^{2/(2+r)}\right) 
\end{equation}
and, for any $\alpha\in (0,2)$,
\begin{equation}\label{meta2}
\widetilde{\Phi}_{\mathbf S_{2,r},\rho_\alpha}(n)\simeq 
\exp\left(- n^{r/(\alpha+r)}[\log n]^{\alpha/(\alpha+r)}\right).
\end{equation}
Further, for $a=(\alpha_1,\dots,\alpha_r)\in (0,2)^r$, let
$\boldsymbol{\mu}_a$ be defined by {\em (\ref{pushforward})}
with $p_i(m)=c_i(1+|m|)^{-1-\alpha_i}$. Let
$\mu_a$ be the probability measure on $\mathbf S_{2,r}$ associated 
to $\boldsymbol{\mu}_a$ by {\em (\ref{defmu})}.  Then we have
\begin{equation}\label{meta3}
\mu_a^{n}(e)\simeq 
\exp\left( -  n^{r/(r+\alpha)}[\log n]^{\alpha/(r+\alpha)}\right)
\end{equation}
where
$$\frac{1}{\alpha}=\frac{1}{r}\left(\frac{1}{\alpha_1}+\dots+
\frac{1}{\alpha_r}\right).$$
\end{theo}
\begin{proof} The lower bound in (\ref{meta1}) follows from Theorem 
\ref{theo-fol} 
(in particular, Example \ref{exa-pol}). The lower bound in (\ref{meta2}) 
then follows from \cite[Theorem 3.3]{Bendikov}. The lower bound in (\ref{meta3})
is Corollary \ref{cor-low}. If we consider the measure $\mu_a$ with $a=(\alpha,\alpha,\dots,\alpha)$, $\alpha\in (0,2)$, it is easy to check that this measure
satisfies
$$\sup_{s>0}\left\{ s\mu_a( g: (1+|g|)^{\alpha}> s)\right\} <\infty,$$
that is, has finite weak $\rho_\alpha$-moment with $\rho_\alpha(g)=(1+|g|)^
\alpha$. This implies that $\mu^{(2n)}_a(e)$ provides an upper bound 
for $\widetilde{\Phi}_{\mathbf S_{2,r},\rho_\alpha}(n)$. See the definifition
of $\widetilde{\Phi}_{G,\rho}$ in Section \ref{sec-otherinv} and \cite{Bendikov}.
The upper bound in (\ref{meta2}) is thus a consequence of the upper bound in
(\ref{meta3}).

We are left with proving the upper bounds contained in 
(\ref{meta1})-(\ref{meta3}).  The proofs follow the same line of reasoning
and we focus on the upper bound (\ref{meta3}).

\begin{lem}
Set $\Gamma =\left\langle
s_{1}^{2},...,s_{r}^{2}\right\rangle < \mathbf S_{2,r}$ 
and $\boldsymbol{\rho }=[\mathbf{s}_{1},%
\mathbf{s}_{2}]\in \mathbf F_r$. The pair $(\Gamma,\boldsymbol \rho)$ 
is an exclusive pair in the sense of {\em Definition \ref{def-exclu}}
\end{lem}
\begin{proof}This was already observed in Example \ref{meta-exclu}.\end{proof}

In order to apply Proposition \ref{Upper} to the pair $(\Gamma,\boldsymbol \rho)$, we now define an appropriate measure 
$\varphi$ on the subgroup $\Gamma=\langle s_1^2,\dots, s_r^2\rangle$ 
of $\mathbf S_{2,r}=\mathbf F_r/[N,N]=\Gamma_2(N)$, 
$N=[\mathbf F_r,\mathbf F_r]$. In this context, $\overline{\Gamma}=(2\mathbb Z)^r\subset \mathbb Z^r=\Gamma_1(N)$. The measure $\varphi$ is simply given by 
\[
\varphi(g)=\sum_{i=1}^{r}\frac{1}{r}\sum_{m\in 
\mathbb{Z}
}c_i(1+|m|)^{-1-\alpha_i}\mathbf{1}_{\{s_{i}^{2m}\}}(g). 
\]%
With this definition, it is  clear that, on $\mathbf S_{2,r}$, we have the Dirichlet 
form comparison
\[
\mathcal{E}_{\mu _{a}}\ge c \mathcal{E}_{\nu \ast \varphi \ast
\nu }. 
\]%
Then by Proposition \ref{Upper},%
\[
\mu _{a}^{\ast n}(e_*)\text{\ }\preceq (\eta \star \overline{\varphi }\star \eta )^{\star n}(e_\star). 
\]%
Here as in the previous section, $*$ denotes convolution in 
$\Gamma_2(N)$ and $\star$ denotes convolution on  
$\mathbb Z\wr \overline{\Gamma} $ (or $\mathbb Z\wr \Gamma_1(N)$). 
Here, $\overline{\Gamma}=(2\mathbb Z)^r$ 
which is a subgroup of (but also isomoprhic to) $\Gamma_1(N)=\mathbb Z^r$. 
The switch-walk-switch measure $q=\eta\star \overline{\varphi}\star\eta$ 
on $\mathbb Z\wr (2\mathbb Z)^r$ has been studied by the authors in 
\cite{SCZ-dv1} where it is proved that
\[
q^{\ast n}(e)\leq \exp \left( - c n^{\frac{r}{r+\alpha }}(\log n)^{%
\frac{\alpha }{r+\alpha }}\right) . 
\]
The proof of this result given in \cite{SCZ-dv1} is based on an extension of 
the Donsker-Varadhan Theorem regarding the  Laplace transform of the number 
of visited points.  This extension treats random walks on $\mathbb Z^r$ driven by
measures that are in the domain of normal attraction of an operator stable law.
See \cite[Theorem 1.3]{SCZ-dv1}. \end{proof}

\subsection{Miscellaneous applications}\label{sec-misc}
This section describes further examples of the
results of Sections \ref{sec-exclu}--\ref{sec-up}. Namely, we consider a number 
of examples consisting of a group $G=\Gamma_1(N)=\mathbf F_r/N$ 
given by an explicit presentation. We identify an exclusive pair 
$(\Gamma,\mathbf \rho)$
with the property that the subgroup $\overline{\Gamma}$ of $\Gamma_1(N)$ 
is either isomorphic to $\Gamma_1(N)$ or has a similar structure so that
$\Phi_{\mathbb Z^d\wr \Gamma_1(N)}\simeq \Phi_{\mathbb Z\wr \overline{\Gamma}}$.
In  each of these examples, the results of Sections 
\ref{sec-low1}-\ref{sec-low2} and those of Section \ref{sec-exclu}-\ref{sec-up}
provide matching lower and upper bounds for $\Phi_{\Gamma_2(N)}$ where $\Gamma_2(N)=\mathbf F_2/[N,N]$.

\begin{exa}[The lamplighter {$\mathbb Z_2\wr \mathbb Z=
\langle a,t\,|\, a^2,[a,t^{-n}at^n],n\in \mathbb Z\rangle$}] In the lamplighter description of $\mathbb Z_2\wr \mathbb Z$, multiplying by  $t$ on the right produces a 
translation of the lamplighter by one unit. Multiplying by $a$ on the right switch the light at the current position of the lamplighter. Let $\Gamma$ 
be the subgroup of $\Gamma_2$ generated by the images of $a$ and $t^2$ and note that $\overline{\Gamma}$ is, in fact, isomorphic to $\Gamma_1(N)$. 
Let $\boldsymbol \rho= [a, t^{-1}at]=  a^{-1} t^{-1}a^{-1}t at^{-1} a t$.
In order to apply Lemma \ref{lem-exclu}, set $\mathbf u= a^{-1} t^{-1}a^{-1}t at^{-1}$, $\mathbf s=a$ and  $\mathbf v=t$ so that $\boldsymbol \rho=
\mathbf u \mathbf s \mathbf v$. By inspection, 
$\mathfrak f_{\boldsymbol \rho}((\bar{u},\overline{us},\mathbf s))\neq 0$ 
(condition (1) of Lemma \ref{lem-exclu}). 
Also, because the elements of $\overline{\Gamma}$ can only have lamps on and the lamplighter at even positions, 
one checks that $\mathfrak f_{\boldsymbol \rho }((x\bar{u},x\overline{us},\mathbf s))=0$ if $x\in \overline{\Gamma}$ (condition (2) of Lemma \ref{lem-exclu}).  
For the same reason, it is clear that
$\mathfrak f_x((\bar{u},\overline{us},\mathbf s))=0$ if 
$x\in \overline{\Gamma}$, 
that is, $(\bar{u},\overline{us},\mathbf s)\not\in U$ (condition (3) 
of Lemma \ref{lem-exclu}).   By the Magnus embedding and 
\cite[Theorem 1.3]{Pittet2000}, we have
$$\Phi_{\Gamma_2(N)}(n)\ge c \Phi_{\mathbb Z^r\wr \Gamma_1(N)}(kn).$$
Applying Lemma \ref{lem-exclu}, Proposition \ref{Upper}, and the fact that 
$\overline{\Gamma}\simeq \Gamma_1(N)$, yields
$$\Phi_{\Gamma_2(N)}(kn)\le C \Phi_{\mathbb Z \wr \Gamma_1(N)}(n).$$
The results of \cite{Erschler2006} gives
$$\Phi_{\mathbb Z^r\wr \Gamma_1(N)}(n)\simeq 
\Phi_{\mathbb Z\wr \Gamma_1(N)}(n)\simeq \exp\left(- n/[\log n]^2\right).$$
Hence we conclude that, in the present case where 
$$\Gamma_1(N)= \mathbb Z_2\wr \mathbb Z=
\langle a,t\,|\, a^2,[a,t^{-n}at^n],n\in \mathbb Z\rangle,$$
we have
$$\Phi_{\Gamma_2(N)}(n)\simeq \exp\left(- n/[\log n]^2\right).$$
This extend immediately to 
$\mathbb Z_q\wr \mathbb Z= \langle a,t\,|\, a^q,[a,t^{-n}at^n],n\in 
\mathbb Z\rangle$. It also extend to similar presentations of 
$F\wr \mathbb Z$ with $F$ finite. See the next class of examples.
\end{exa}

\begin{exa}[Examples of the type {$K\wr \mathbb Z^d$}]
Let $K=\langle \mathbf k _1,\dots,\mathbf k_m\,|\, N_K\rangle $ 
be a $m$ generated group. The wreath product $K\wr \mathbb Z^d$ admits the 
presentation $\mathbf F_r/N$  with $r=m+d$ generators denoted by 
$$\mathbf k_1,\dots,\mathbf k_m,\mathbf t_1,\dots,\mathbf t_d$$
and relations
$[\mathbf t_i,\mathbf t_j], 1\le i,j\le d$, $ N_K$ and
$$[\mathbf k',
\mathbf t^{-1} \mathbf k \mathbf t], \mathbf k,\mathbf k'\in \mathbf F(\mathbf k_1,\dots,\mathbf k_m), \mathbf t= \mathbf t_1^{x_1}\cdots \mathbf t_d^{x_d},\;(x_1,\dots,x_d)\neq 0
.$$
Without loss of generality, we can assume that the image of $\mathbf k_1$ 
in $K$ is not trivial. Let $\Gamma$ be the subgroup of $\Gamma_2(N)$ generated 
by the images of $\mathbf t_i^2$, $1\le i\le d$. Let 
$$\boldsymbol \rho =[\mathbf k_1, \mathbf t_1^{-1}\mathbf k_1 \mathbf t_1]$$
and write
$$\boldsymbol \rho = \mathbf u \mathbf s \mathbf v 
\mbox{ with } \mathbf u= \boldsymbol\rho \mathbf t_1^{-1} \mathbf k_1^{-1}, \mathbf s=\mathbf k_1, \mathbf v=\mathbf t_1.$$
As in the previous example, $(\Gamma, \boldsymbol \rho)$ is an exclusive pair 
and $\overline{\Gamma}$ is in fact isomorphic to $\Gamma_1(N)$. By the 
same token, it follows that
$$\Phi_{\Gamma_2(N)}(n)\ge c \Phi_{\mathbb Z^r\wr \Gamma_1(N)}(kn)
\mbox{ and }\Phi_{\Gamma_2(N)}(kn)\le C \Phi_{\mathbb Z \wr \Gamma_1(N)}(n).$$

Now, thanks to the results of \cite{Erschler2006} concerning wreath products, 
we obtain
\begin{itemize}
\item If $K$ is a non-trivial finite group then 
$$\Phi_{\Gamma_2(N)}(n)\simeq \exp\left(-n/[\log n]^{2/d}\right).$$
\item If $K$ is not finite but has polynomial volume growth then
$$\Phi_{\Gamma_2(N)}(n)\simeq \exp\left(-n \left(\frac{\log \log n}{\log n}\right)^{2/d}\right).$$
\item If $K$ is polycyclic with exponential volume growth then
$$\Phi_{\Gamma_2(N)}(n)\simeq \exp\left(-n/[\log \log n]^{2/(d+1)}\right)$$
\end{itemize}
In particular, when $\Gamma_1(N)=\mathbb Z\wr \mathbb Z$ with presentation
$\langle a,t\mid \lbrack a,t^{-n}at^{n}],n\in 
\mathbb{Z}\rangle $ we obtain that
$$\Phi_{\Gamma_2(N)}(n)\simeq 
\exp \left( - n\left(\frac{\log \log n}{\log n}\right)^{2}\right) . $$
\end{exa}

\begin{exa}[The Baumslag-Solitar group {$\mbox{BS}(1,q)$}] Consider the 
presentation
\[
\mbox{BS}(1,q)=\Gamma_1(N)=\mathbf F_2/N
=<a,b\mid a^{-1}ba=b^{q}> 
\]%
with $q>1$.  In order to apply Proposition \ref{Upper}, let 
$\Gamma$ be the group generated by the image of $a^2$ and $b$ in 
$\Gamma_2(N)$. Let $\boldsymbol \rho= b^{-q}a^{-1}ba$, 
$\mathbf u= b^{-q}a^{-1}$, $\mathbf s=b$, $\mathbf v=a$. One checks that 
$(\Gamma,\boldsymbol \rho)$ is an exclusive pair and that 
$\overline{\Gamma}\simeq \mbox{BS}(1,q^2)$. After some computation, we obtain
\[
\Phi _{\Gamma_2(N)}(n)\simeq \exp \left( - n/[\log n]^2
\right) . 
\]
\end{exa}

\begin{exa}[Polycyclic groups] Let $G$ be a polycyclic group with 
polycyclic series $G=G_1 \rhd G_2\rhd \dots\rhd G_{r+1}=\{e\}$, $r\ge 2$.
For each $i$, $1\le i\le r$, let $a_i$ be an element in $G_i$ whose 
projection in $G_i/G_{i+1}$ generates that group.  Write $G=\mathbf F_r/N$
where $\mathbf s_i$ is sent to $a_i$. This corresponds to the
standard polycyclic presentation of $G$ relative to $a_1,\dots, a_n$ and
$N$ contains a word of the form
$$\boldsymbol \rho= \mathbf s_1^{-1}\mathbf s_2 \mathbf s_1 
\mathbf s_r^{\alpha_r}\cdots \mathbf s_2^{\alpha_2}$$
where $\alpha_\ell$, $2\le \ell\le r$ are integers. 
See \cite[page 395]{Sims}.

\begin{theo} \label{th-polyc}
Let $G=\Gamma_1(N)$ be an infinite polycyclic group equipped with a polycyclic 
presentation as above with at least two generators. 
\begin{itemize}
\item If $G$ has polynomial volume growth of degree $D$, then
$$\Phi_{\Gamma_2(N)}(n)\simeq \exp\left(- n^{D/(2+D)} [\log n]^{2/(2+D)}\right).$$\item If $G$ has exponential volume growth then
$$\Phi_{\Gamma_2(N)}(n)\simeq \exp\left(- n/[\log n]^2\right).$$
\end{itemize}
\end{theo}
\begin{proof}
Our first step is to construct an exclusive pair $(\Gamma,\boldsymbol \rho)$ with
$\overline{\Gamma}=\bar{\pi}(\Gamma)$ of finite index in $\Gamma_1(N)$.

Assume first that $G_1/G_2$ is finite. In this case, let
$\Gamma=\langle s_2,\dots, s_r\rangle$. Assume that $x\in \Gamma$ is such that
$\mathfrak f_{\boldsymbol \rho}((\bar{x}\bar{s}_1^{-1},
\bar{x}\bar{s}_1^{-1}\bar{s}_2,\mathbf s_2))\neq 0$. 
Then there must be a prefix $\mathbf u$ of $\boldsymbol \rho$ such that
$\pi(\mathbf u)=\bar{x}\bar{s}_1^{-1}$. Computing modulo $\pi(\Gamma)=G_2$, 
the only prefixes of $\boldsymbol \rho$ that can have this property are 
$\mathbf s_1^{-1}$ and $\mathbf s_1^{-1}\mathbf s_2$.  If $\mathbf u=\mathbf s_1^{-1}$
then $\bar{x}$ is the identity. If $\mathbf u=\mathbf s_1^{-1}\mathbf s_2$
then  $\mathbf s_1^{-1} \mathbf s_2^2$ is not a prefix of $\boldsymbol \rho$ 
and $\mathfrak f_{\boldsymbol \rho} ((\bar{x}\bar{s}_1^{-1},
\bar{x}\bar{s}_1^{-1}\bar{s}_2,\mathbf s_2))= 0$, a contracdiction. 
It follows that condition 2 of Lemma \ref{lem-exclu} is satisfied.
In this case, it is obvious that condition 3 holds as well. Further,
$\pi(\Gamma)=G_2$ is a  subgroup of finite index in $G=\Gamma_1(N)$.  

In the case when $G_1/G_2\simeq \mathbb Z$, set
$$\Gamma= \langle s_1^2, s_2,\dots, s_r\rangle < \Gamma_2(N).$$
The same argument as used in the case when $G_1/G_2$ is finite apply 
to see that
condition 2 of Lemma \ref{lem-exclu} is satisfied. To check that condition 3 of Lemma \ref{lem-exclu} is satisfied, observe that, 
if $\mathfrak f_g((\bar{y},\bar{y}\bar{s}_i,\mathbf s_i))\neq 0$ 
with $2\le i\le r$ and $g\in \Gamma$ then
$\bar{y}$ must belong to $\overline{\Gamma}$. But, by construction 
$\bar{s}_1^{-1}\not\in \overline{\Gamma}$. Therefore $\mathfrak f_g(\bar{s}_1^{-1}, \bar{s}_1^{-1}\bar{s}_2, \mathbf s_2)=0$ for every $g\in \Gamma$. Finally, 
$\overline{\Gamma}$ is obviously of index $2$ in $\Gamma_1(N)$.

By the Magnus embedding we have $ c\Phi_{\mathbb Z^r\wr \Gamma_1(N)}(kn)\le \Phi_{\Gamma_2(N)}(n)$.
By Theorem \ref{Upper} and the existence of the exclusive pair $(\Gamma,\rho)$ exhibited above, 
we also have
$c\Phi_{\Gamma_2(N)}(kn)\le \Phi_{\mathbb Z \wr \overline{\Gamma}}(n)$ with 
$\overline{\Gamma}$ of finite index in $\Gamma_1(N)$. Because $\Gamma_1(N)$ is infinite polycyclic,
the desired result follows from the known results about wreath products. See \cite{Erschler2006}.

\end{proof}
\end{exa}

\section{Iterated comparison and $\mathbf S_{d,r}$ with $d>2$}
\setcounter{equation}{0}

Let $\mathbf F_r/N=\Gamma_1(N)$ be a given presentation. Write $N^{(2)}=[N,N]$
and $N^{(\ell)}=[N^{(\ell-1)},N^{(\ell-1)}]$, $\ell>2$.  The goal of this 
section is to obtain bounds on the probability of return for random walks on 
$\Gamma_\ell(N)=\mathbf F_r/N^{(\ell)}$. Our approach is to iterate the method
developped in the previous sections in the study of random walks on 
$\Gamma_2(N)$.

We need to fix some notation. We will use $*=*_\ell$ to denote convolution in 
$\Gamma_{\ell}(N)$. In general, $\ell$ will be fixed so that 
there will be no need to distinguish between different $*_\ell$. 
We will consider several wreath products $A\wr G$ as well as  
iterated wreath products
$$A\wr(A\wr(\cdots (A\wr G)\cdots))$$ where $A$ and $G$ are given with 
$A$ abelian (in fact, $A$ will be either $\mathbb Z$ or $\mathbb Z^r$). 
Set  $W(A,G)=W_1(A,G)= A\wr G$ and $W_k(A,G)= W(A,W_{k-1}(A,G))$. 
Depending on the context, we will denote convolution in $W_k(A,G)$ by
$$\star_k \mbox{ or } \star_{W_k}\mbox{ or } \star_{W_k(A,G)}.$$

Let $\mu$ be a probability measure on $G$ and $\eta$ 
be a probability measure on $A$. Note that the measures $\mu$ and $\eta$ can also be viewed, in a natural way, as measures on $W(A,G)$ with $\eta$ 
being supported by the copy of $A$ that sits above the identity element of $G$
in $A\wr G$.  
The associated switch-walk-switch measure 
on $W=W_1(A,G)$ is the measure 
$$q=q_1(\eta,\mu)=\eta\star_1 \mu\star_1 \eta.$$
Iterating this construction, we define the probability measure $q_k$ on 
$W_k(A,G)$ by the iterative formula
$$q_k= q_k(\eta,\mu) =\eta\star_k q_{k-1}\star_k \eta .$$
We refer to $q_k$ as the iterated swith-walk-switch measure on $W_k$ 
associated with the initial pair $\eta,\mu$. We will make repeated use of the 
following simple lemma.

\begin{lem}\label{lem-push1}
Let $A,G,H$ be finitely generated groups with $A$ abelian. 
Let $\theta :G\rightarrow
H$ be  a group homomorphism. Define  $\theta_1: W_1(A,G)\ra W_1(A,H)$ by
\[
\theta _1:(f,x)\mapsto (\overline{f},\theta (x)), 
\mbox{ where }
\overline{f}(h)=\sum_{g:\theta (g)=h}f(g) 
\]%
with the convention that sum over empty set is $0.$ 
Then $\theta_1$ is group homorphism. 

Define  $\theta_k: W_k(A,G)\ra W_k(A,H)$ 
by iterating the previous construction so 
$$\theta_k=(\theta_{k-1})_1 :W_1(A,W_{k-1}(A,G))\ra W_1(A,W_{k-1}(A,H)).$$
Then  $\theta_k$ is group homorphism. 
Moreover, if $\theta $
is injective (resp., surjective), then $\theta_k$ is also
injective (resp., surjective). 
\end{lem}
\begin{proof} The stated conclusions follow  by inspection. \end{proof}

\begin{lem} \label{lem-push2}
Let $A,G,H$ be finitely generted groups with $A$ abelian. 
Let $\mu $ and $\eta $ 
be a probability measures on $G$ and $A$, respectively. Let $\theta:G\ra H$ be a
homorphism and $\theta_k:W_k(A,G)\ra W_k(A,H)$ be as in 
{\em Lemma \ref{lem-push1}}. Let $\theta_k(q_k(\eta,\mu))$ be the pushforward 
of the iterated switch-walk-switch measure $q_k(\eta,\mu)$ on $W_k(A,G)$ under $\theta_k$. Then we have
$$\theta_k (q_k(\eta,\mu)) =q_k(\eta,\theta(\mu)).$$  
\end{lem}
\begin{proof}It suffices to check the case $k=1$ 
where the desired conclusion reads
$$\theta_1( \eta\star_{A\wr G}\mu\star_{A\wr G}\eta)=
\eta\star_{A\wr H}\theta(\mu)\star_{A\wr H} \eta.$$
This equality follows from the three identities
$$\theta_1( \eta\star_{A\wr G}\mu\star_{A\wr G}\eta)
=\theta_1( \eta)\star_{A\wr H}\theta_1(\mu)\star_{A\wr G}\theta_1(\eta),$$
$$\theta_1(\mu)=\theta(\mu) \mbox{ and }\;\;\theta_1(\eta)=\eta.$$
The first identity holds because $\theta_1$ is an homomorphism. 
The other two identities hold by inspection (with some slight abuse 
of notation). 
\end{proof}

\subsection{Iterated lower bounds}

This section develops lower bounds for the probability of return 
of symmetric finitely supported random walks on 
$\Gamma_\ell(N)=\mathbf F_r/N^{(\ell)}$. By Dirichlet form comparison techniques
(see \cite{Pittet2000}), it suffices to consider  the case of the measure 
$\mu_\ell$ on $\Gamma_\ell(N)$ which is the image under the projection 
$\pi_\ell:\mathbf F_r\ra \mathbf F_r/N^{(\ell)}$  
 of the lazy symmetric simple 
random walk measure $\boldsymbol \mu$ on $\mathbf F_r$ defined at
(\ref{lazysr}), that is $\mu_\ell=\pi_\ell(\boldsymbol \mu)$.  
On $\mathbb Z^r$, let the probability $\eta$ be defined by 
$\eta(\pm \epsilon_i)=1/2r$ where $(\epsilon_i)_1^r$ is the canonical 
basis for $\mathbb Z^r$. 
Let $q_{\ell,j}$ be the $j$-th
iterated switch-walk-switch measure on $W_j(\mathbb Z^r, \Gamma_{\ell-j}(N))$
based on the probability measures $\eta$ on $\mathbb Z^r$ and 
$\mu_{\ell-j}$ on $\Gamma_{\ell-j}(N)$.  

\begin{theo} \label{theo-itlb}
Let the presentation $\Gamma_1(N)=\mathbf F_r/N$ be given.
Fix an integer $\ell$ and let $\mu_\ell=\pi_\ell(\boldsymbol \mu) $ be 
the probability measure on $\Gamma_\ell(N)$ describe above.
Let $*$ denote convolution on $\Gamma_\ell(N)$ and $\star$ denote convolution on 
$W_{\ell-1}(\mathbb Z^r,\Gamma_1(N))$. Then there exist $c,k\in (0,\infty)$ such 
that
$$\forall\,n,\;\;\mu_{\ell}^{*2n}(e_*)\ge c q_{\ell,\ell-1}^{\star 2kn}(e_\star).$$
\end{theo} 
\begin{proof} The proof is obtain by an iterative procedure based on 
repeated use of the Magnus embedding $\Gamma_m(N)\hookrightarrow \mathbb Z^r \wr 
\Gamma_{m-1}(N)$ and comparison of Dirichlet forms. The desired conclusion follows immediately from the following two lemmas.
\end{proof}

\begin{lem} Let $*$ denotes convolution on $\Gamma_\ell(N)$.
Let $\star$ denote convolution on $W(\mathbb Z^r,\Gamma_{\ell-1}(N))=
\mathbb Z^r\wr \Gamma_{\ell-1}(N)$.
Let $\mu_{\ell}$ and $\mu_{\ell-1}$
be as defined above. 
Then
$$\mu_\ell^{*2n}(e_*)\ge 
c [\eta \star \mu_{\ell-1} \star\eta]^{\star 2kn}(e_\star).$$
\end{lem}
\begin{proof} Let $\bar{\psi}_\ell :\Gamma_{\ell}(N)\hookrightarrow \mathbb Z^r\wr \Gamma_{\ell-1}(N)$ be the Magnus embedding. 
Then $\mu_\ell^{*n}(e_*)=[\bar{\psi}_\ell(\mu)]^{\star n}(e_\star)$
and, by a simple Dirichlet form comparison argument,
$$[\bar{\psi}_\ell(\mu)]^{\star 2n}(e_\star)\ge c [\eta\star\mu_{\ell-1}\star \eta ]^{\star 2kn}(e_\star).$$
\end{proof} 

\begin{lem} Fix two integers $0< j< \ell$.  
Let $\star_j$ denote convolution on 
$W_j(\mathbb Z^r, \Gamma_{\ell-j}(N))$.
Then, for $2\le j<\ell$,
$$q_{\ell,j-1}^{\star_{j-1} 2n}(e_{\star_{j-1}})
\ge c q_{\ell,j}^{\star_j 2kn}(e_{\star _j}).$$ 
\end{lem}
\begin{proof} By definition, we have 
$$q_{\ell,j-1}=\eta \star_{j-1}
q_{\ell-1,j-2} \star_{j-1} \eta .$$  
where $q_{\ell-1,j-2}$ is the switch-walk-switch measure on 
$W_{j-1}(\mathbb Z^r,\Gamma_{\ell-j+1}(N))$. Let $\bar{\psi}$ 
denote the Magnus embedding
$$\bar{\psi}: \Gamma_{\ell-j+1}(N) 
\hookrightarrow \mathbb Z^r\wr \Gamma_{\ell-j}(N).$$ 
Let 
$$\widetilde{\psi}: W_{j-1}(\mathbb Z^r,\Gamma_{\ell-j+1}(N))
\hookrightarrow W_{j-1}(\mathbb Z^r, \mathbb Z^r\wr \Gamma_{\ell-j}(N))= 
W_j(\mathbb Z^r, \Gamma_{\ell-j}(N)
) $$ 
its natural extension as in Lemma \ref{lem-push1}.
Observe that
\begin{eqnarray*}q_{\ell,j-1}^{\star_{j-1}2n}(e_{\star_{j-1}})&=&
\widetilde{\psi}(\eta \star_{j-1}
q_{\ell-1,j-2} \star_{j-1} \eta)^{\star_j 2n}(e_{\star_j})\\
&=& [\eta \star_j \widetilde{\psi}(q_{\ell-1,j-2}) \star_j \eta]^{\star_j 2n}(e_{\star_j})  
\end{eqnarray*}
where we used Lemma \ref{lem-push2} to obtain the second  identity.
Again, by a simple Dirichlet form comparison argument,
\begin{eqnarray*}
[\eta \star_j \widetilde{\psi}( q_{\ell-1,j-2})\star_j \eta]^{\star_j 2n}(e_{\star_j})&\ge& c 
[\eta \star_j  q_{\ell-1,j-1}\star_j \eta]^{\star_j 2kn}(e_{\star_j})    \\
&=& c q_{\ell,j}^{\star_j 2kn}(e_{\star_j}).
\end{eqnarray*} 
\end{proof}

Propositions \ref{prob-FC}--\ref{wr-FC}--\ref{pro-comp} (which are based on the results in \cite{CGP,Erschler2006}) provide 
 us with good lower bounds for the probability of return on iterated 
wreath poduct.  Namely,
\begin{itemize}
\item Assume that $A=\mathbb Z^b$ with $b\ge 1$ and 
$G$ has polynomial volume growth of degree $D$. Then, for $\ell\ge 2$,
$$\Phi_{W_\ell(A,G)}(n)\simeq \exp \left(- n \left(\frac{ \log_{[\ell]} n}{\log _
{[\ell-1]} n }\right)^{2/D}\right).$$
\item Assume that $A=\mathbb Z^b$ with $b\ge 1$ and 
$G$ is polycyclic  with exponential volume growth. Then, for $\ell\ge 1$,
$$\Phi_{W_\ell(A,G)}(n)\simeq \exp \left(- n /[\log_{[\ell]}n]^2\right).$$
This applies, for instance, when $G$ is the Baumslag-Solitar group $\mbox{BS}(1,q)$, $q>1$. Further, the same result holds for the wreath product  
$G=\mathbb Z_q\wr \mathbb Z$, $q>1$, (even so it is not polycyclic).
\end{itemize}

Together with Theorem \ref{theo-itlb}, 
these computations yield the following results.

\begin{cor} \label{cor-itlow}
Let $\Gamma_\ell(N)=\mathbf F_r/N^{(\ell)}$.
\begin{itemize}
\item Assume that $\Gamma_1(N)$ has polynomial volume growth of degree $D$.
Then, for $\ell\ge 3$,
$$\Phi_{\Gamma_\ell(N)}(n)\geq 
\exp \left(- Cn \left(\frac{ \log_{[\ell-1]} n}
{\log _{[\ell-2]} n }\right)^{2/D}\right).$$
\item Assume that $\Gamma_1(N)$ is $\mbox{\em BS}(1,q)$ with $q>1$, or $\mathbb Z_2\wr \mathbb Z$, or polycyclic of exponential volmue growth volume growth. Then, for $\ell\ge 2$,
$$\Phi_{\Gamma_\ell(N)}(n)\geq 
\exp \left(- C n /[\log_{[\ell-1]}n]^2\right).$$
\item Assume that $\Gamma_1(N)= K\wr \mathbb Z^D$, $D\ge 1$ and $K$ finite. Then,
for $\ell\ge 2$,
$$\Phi_{\Gamma_\ell(N)}(n)\geq 
\exp \left(- Cn/
[\log _{[\ell-1]} n]^{2/D}\right).$$
\item Assume that $\Gamma_1(N)=\mathbb Z ^a\wr \mathbb Z^D$, $a,D\ge 1$. Then,
for $\ell\ge 2$,
$$\Phi_{\Gamma_\ell(N)}(n)\geq 
\exp \left(- Cn \left(\frac{ \log_{[\ell]} n}
{\log _{[\ell-1]} n }\right)^{2/D}\right).$$
\end{itemize}
\end{cor}

\subsection{Iterated upper bounds}

We now present an iterative approach to obtain upper bounds on 
$\Phi_{\Gamma_\ell(N)}$. Although similar in spirit to the iterated 
lower bound technique developed in the previous section, 
the iterative upper bound 
method is both more difficult and mcuch less flexible. In the end, we will 
be able to apply it only in the case of the free 
solvable groups $\mathbf S_{d,r}$, that is, when $N=[\mathbf F_r,\mathbf F_r]$.

Our first task is to formalize algebraically the content of 
Proposition \ref{Upper}. Recall once more that the Magnus embedding
provides an injective homomorphism $\bar{\psi}: \mathbf F_r/[N,N]
\hookrightarrow 
\left(\sum_{x\in \mathbf F_r/N}\mathbb Z^r_x\right)\rtimes \mathbf F_r/N$
with $\bar{\psi}(g)=(\bar{a}(g),\bar{\pi}(g))$.  Let $\Gamma$ be a subgroup of 
$\mathbf F_r/[N,N]$ and $\boldsymbol \rho\in N\setminus [N,N]\subset 
\mathbf F_r$. Set $\rho=\pi_2(\boldsymbol \rho)$ and 
$\overline{\Gamma}=\pi(\Gamma)\subset \mathbf F_r/N$. 

Assume that $(\Gamma,\boldsymbol \rho)$ is an exclusive pair as in Definition
\ref{def-exclu}. We are going to construct a surjective homomorphism
$$\vartheta: \langle \Gamma,\rho\rangle\ra \mathbb Z\wr \overline{\Gamma}.$$
Let $g\in \langle \Gamma,\rho\rangle$. Consider two decompositions of $g$ as products
$$g=
\gamma_1\rho^{x_1}\gamma_2\rho^{x_2}\cdots \gamma_p \rho^{x_p}\gamma_{p+1}=
\gamma'_1\rho^{x'_1}\gamma'_2\rho^{x'_2}\cdots  \gamma'_q 
\rho^{x'_q}\gamma'_{q+1}$$
with $\gamma_i\in \Gamma$, $1\le i\le p+1$, $\gamma'_i\in \Gamma$, $1\le i\le q+1$.  Set 
$\sigma_i= \gamma_1\dots \gamma_i$, $1\le i\le p+1$, and
$\sigma'_i= \gamma_1\dots \gamma_i$, $1\le i\le q+1$. 
Observe that
$$g=
\sigma_1\rho^{x_1}\sigma_1^{-1}\sigma_2\rho^{x_2}\sigma_2^{-1}\cdots  \sigma_p \rho^{x_p}\sigma_p^{-1}\sigma_{p+1}= \alpha \sigma_{p+1} $$
where 
$$\alpha= \sigma_1\rho^{x_1}\sigma_1^{-1}\sigma_2\rho^{x_2}\sigma_2^{-1}\cdots  \sigma_p \rho^{x_p}\sigma_p^{-1}.$$
Similarly $g=\alpha' \sigma'_{q+1}$ and we have
$$(\alpha')^{-1}\alpha = \sigma'_{q+1}(\sigma_{p+1})^{-1}.$$
By Lemma \ref{lem-alg-a}, 
we have
\begin{eqnarray*}\bar{a}(g)&= &
\sum_1^p x_i\tau_{\bar{\sigma}_i}\bar{a}(\rho) +\bar{a}(\sigma_{p+1})\\
&=& \sum_1^q x'_i\tau_{\bar{\sigma}'_i}\bar{a}(\rho) +\bar{a}(\sigma'_{q+1}).
\end{eqnarray*}
and
$$
\sum_1^p x_i\tau_{\bar{\sigma}_i}\bar{a}(\rho) - \sum_1^q x'_i\tau_{\bar{\sigma}'_i}\bar{a}(\rho) 
=\bar{a}(\sigma'_{q+1}\sigma_{p+1}^{-1}).
$$
As $(\Gamma,\boldsymbol \rho)$ is an exclusive pair, condition (ii) of 
Definition \ref{def-exclu} implies that
$$
\sum_1^p x_i\tau_{\bar{\sigma}_i}\bar{a}(\rho) - \sum_1^q x'_i\tau_{\bar{\sigma}'_i}\bar{a}(\rho) 
=\bar{a}(\sigma'_{q+1}\sigma_{p+1}^{-1})=0.
$$
Hence
$$
\sum_1^p x_i\tau_{\bar{\sigma}_i}\bar{a}(\rho) 
= \sum_1^q x'_i\tau_{\bar{\sigma}'_i}\bar{a}(\rho)$$
in $\sum_{x\in \Gamma_1(N)}\mathbb Z^r_x$. This also implies that
$a(\sigma_{p+1})=a(\sigma_{q+1}')$. By construction, we also have 
$\bar{\pi}(\sigma_{p+1})=\bar{\pi}(\sigma'_{q+1})$. Hence, 
$\sigma_{p+1}=\sigma'_{q+1}$ in $\Gamma$.

By condition (i) of Definition \ref{def-exclu} (see Remark \ref{rem-wr-rho}),
we can identify $$
\sum_1^p x_i\tau_{\bar{\sigma}_i}\bar{a}(\rho) 
= \sum_1^q x'_i\tau_{\bar{\sigma}'_i}\bar{a}(\rho)$$
with the
 element 
$$ \left(\sum_1^p x_i \mathbf 1_{h}(\sigma_i)\right)_{h\in 
\overline{\Gamma}}\;\; \mbox{ 
of }\;\;\sum_{h \in \overline{\Gamma}}\mathbb Z_{h}.$$
This preparatory work allows us to define a map
\begin{eqnarray*}
\vartheta &: & \langle \Gamma,\rho\rangle \ra \mathbb Z\wr \overline{\Gamma}\\g&=&\gamma_1\rho^{x_1}\gamma_2\rho^{x_2}\cdots \gamma_p \rho^{x_p}\gamma_{p+1}
\mapsto \left(
\left(\sum_1^p x_i \mathbf 1_{h}(\sigma_i)\right)_{h\in 
\overline{\Gamma}}, \bar{\pi}(g)\right).
\end{eqnarray*}
\begin{lem} \label{lem-vartheta}
The map $\vartheta :  \langle \Gamma,\rho\rangle \ra \mathbb Z\wr \overline{\Gamma}$ is a surjective homomorphism.
\end{lem}
\begin{proof}Note that $\vartheta(e)$ is the identity element in 
$\mathbb Z\wr \overline{\Gamma}$.
To show that $\vartheta$ is an homomorphism, it suffices to check that, for any 
$g\in \langle \Gamma,\rho\rangle$ and  $\gamma\in \Gamma$ 
$$\vartheta( g\gamma)=\vartheta(g)\vartheta(\gamma),\;\;
\vartheta( g\rho^{\pm 1})=\vartheta(g)\vartheta(\rho^{\pm 1}).$$
These identities follow by inspection. One easily check that 
$\vartheta$ is surjective.
\end{proof}

\begin{lem}\label{lem-varthetaconv}
Let $\mu$ be a probability measure supported on 
$\Gamma$ and $\nu$ be the probability measure defined by $\nu(\rho^{\pm1})=1/2$.
Let $\eta$ be the probability measure on $\mathbb Z$ defined by 
$\eta(\pm 1)=1/2$. Let $*$ be convolution on $\langle \Gamma,\rho\rangle<\Gamma_2(N)$ and $\star$ be convolution on $\mathbb Z\wr \overline{\Gamma}$
Then we have
$$\vartheta(\nu*\mu*\nu) = \eta\star \bar{\pi}(\mu)\star \eta.$$ 
\end{lem} 
\begin{proof} This follows from the fact that $\vartheta$ is an 
homomorphism, $\left. \vartheta\right|_\Gamma=\bar{\pi}$ and $\vartheta(\nu)=\eta$.
\end{proof}

In addition to the canonical projections 
$\pi_j: \mathbf F_r\ra \mathbf F_r/N^{(j)}=\Gamma_j(N)$, for $1\le j\le i$, 
we also consider the projection
$\pi^i_j: \Gamma_i(N)\ra \Gamma_j(N).$

\begin{defin} Fix a presentation $\Gamma_1(N)=\mathbf F_r/N$ and an 
integer $\ell$. Let
$\Gamma_i$ be a finitely generated subgroup of 
$\Gamma_{i}(N)$, $2\le i\le \ell$. Set
$$\Gamma'_{i-1}=\pi^{i}_{i-1}(\Gamma_i),\;\; 2\le i\le \ell.$$ 
Let $\boldsymbol \rho_i \in \mathbf F_r$, $2\le i\le \ell$. Set
$\rho_\ell=\pi_\ell(\boldsymbol \rho_\ell)$.
We say that 
$(\Gamma_i,\boldsymbol \rho_i)_2^\ell$ is an exclusive sequence 
(adapted to $(\Gamma_{i}(N))_1^\ell$) if the following properties hold: 
\begin{enumerate}
\item  $\Gamma_\ell < \Gamma_\ell(N)$ and $\pi^\ell_{\ell-1}(\rho_\ell)$ is trivial.
\item For $2\le j\le \ell-1$,
$\Gamma_{j} < \Gamma'_{j}, \;\; 
\rho_j\in \Gamma'_{j}
$ and
$\pi^j_{j-1}(\rho_j)$ is trivial.
\item For each $2\le i\le \ell$, $(\Gamma_i ,\boldsymbol \rho_i)$ 
is an exclusive pair in $\Gamma_2(N^{(i-1)})= \Gamma_{i}(N)$.  
\end{enumerate}
\end{defin}

\begin{theo} \label{theo-itup}
Fix a presentation $\Gamma_1(N)=\mathbf F_r/N$ and an integer 
$\ell\ge 2$. Assume that there exists an exclusive sequence 
$((\Gamma_i,\boldsymbol \rho_i))_2^\ell$ adapted to $(\Gamma_j(N))_1^\ell$. 
Then 
there exists $k,C\in (0,\infty)$ such that
$$\Phi_{\Gamma_\ell(N)} (kn) \le  C\Phi_{W_{\ell-1}(\mathbb Z,\Gamma'_1)}(n)$$
where $\Gamma'_1=\pi^2_1(\Gamma_2)<\Gamma_1(N)$.
\end{theo}
\begin{rem} The technique and results  of \cite{Erschler2006} provides good upper bounds on $\Phi_G$ when $G$ is an iterated wreath product such as 
$W_{\ell-1}(\mathbb Z,\Gamma'_1)$ and we have some information on $\Gamma'_1$. 
The real difficulty in applying the theorem above lies in finding 
an exclusive sequence.
\end{rem}
\begin{proof} The Theorem follows immediately from the following two lemmas. 
\end{proof}
We will need the following notation. For each $1\le i\le \ell$, let $\phi_i$ be a symmetric 
finitely supported probability measure on $\Gamma_ i$ with generating support.
Let $\mu_\ell=\pi(\boldsymbol \mu)$ be the projection on $\Gamma_\ell(N)$ 
of the lazy symmetric simple random walk probability measure on $\mathbf F_r$.
Let $\nu_i$ be the probability measure on $\Gamma_i(N)$ given by 
$\nu_i(\rho_i^{\pm 1})=1/2$.  

\begin{lem}Under the hypothesis of {\em Theorem \ref{theo-itup}},  
there are $k,C\in (0,+\infty)$ such that
$$\Phi_{\Gamma_\ell(N)}(kn)\le C \Phi_{W_1(\mathbb Z,\Gamma'_{\ell-1})}(n).$$
\end{lem}
\begin{proof} For this proof, let $*$ be convolution on $\Gamma_{\ell}(N)$ and 
$\star$ be convolution on $\mathbb Z\wr \Gamma'_{\ell-1}=W_1(\mathbb Z,\Gamma'_{\ell-1})$, $\Gamma'_{\ell-1}=\pi^\ell_{\ell-1}(\Gamma_\ell)$. 
Since $\nu_\ell*\phi_\ell*\nu_\ell$ is symmetric and finitely supported
on $\Gamma_\ell(N)$, we have the Dirichlet form comparison
$$\mathcal E_{\mu_\ell}\ge c\mathcal E_{\nu_\ell*\phi_\ell*\nu_\ell}.$$
Hence
$$\mu_\ell^{*2kn}(e_*) \le C[\nu_\ell*\phi_\ell*\nu_\ell]^{*2n}(e_*).$$ 
Note that the measure $\nu_\ell*\phi_\ell*\nu_\ell$ lives on $\langle \Gamma_\ell,\rho_\ell\rangle < \Gamma_\ell(N)$.  By Lemma \ref{lem-vartheta},
we have the surjective  homorphism
$\vartheta_\ell: \langle \Gamma_\ell,\rho_\ell\rangle \ra \mathbb Z\wr
\Gamma'_{\ell-1}=W_1(\mathbb Z,\Gamma'_{\ell-1})$.  Hence
$$ 
[\nu_\ell*\phi_\ell*\nu_\ell]^{*2n}(e_*)\le 
[\vartheta_\ell(\nu_\ell*\phi_\ell*\nu_\ell)]^{\star2n}(e_{\star} )$$ 
where $e_{\star}$ is the identity element in 
$W_1( \mathbb Z,\Gamma'_{\ell-1})$.  By Lemma \ref{lem-varthetaconv},
$$[\vartheta_\ell(\nu_\ell*\phi_\ell*\nu_\ell)]^{\star 2n}(e_{\star} )
=(\eta *\pi^\ell_{\ell-1}(\phi_\ell)*\eta]^{\star 2n}(e_{\star} ).$$
This shows that $\Phi_{\Gamma_\ell(N)}(kn)\le C \Phi_{W_1(\mathbb Z,\Gamma'_{\ell-1})}(n)$.  
\end{proof} 
\begin{lem}
Under the hypothesis of {\em Theorem \ref{theo-itup}}, for each $j$, $1\le j\le \ell-2$,
there are $k,C\in (0,+\infty)$ such that
$$\Phi_{W_j(\mathbb Z, \Gamma'_{\ell-j})}(kn)\le C 
\Phi_{W_{j+1}(\mathbb Z,\Gamma'_{\ell-j-1})}(n).$$
\end{lem}
\begin{proof} For this proof, we let $\star_j$ denote convolution on 
the iterated wreath product $W_j(\mathbb Z, \Gamma'_{\ell-j})$.
To control $\Phi_{W_j(\mathbb Z,\Gamma'_{\ell-j})}$ from above, 
it suffices to
control from above the probability of return $n\mapsto q_j^{\star_j 2n}(e_
{\star_j})$,
for the iterated switch-walk-switch measure
$q_j$ based on $\eta$ and $\pi^{\ell-j+1}_{\ell-j}(\phi_{\ell-j+1})$.

By a simple comparison of Dirichlet forms on the group 
$W_j(\mathbb Z,\Gamma'_{\ell-j})$, we have
\begin{equation}\label{crux1}
q_j^{\star_j 2kn}(e_
{\star_j})\le C\widetilde{q}_j^{\star_j 2n}(e_
{\star_j})\end{equation}
where $\widetilde{q}_j$ is the iterated switch-walk-switch measure based on 
$\eta$ and $$\nu_{\ell-j}\star_{j}\phi_{\ell-j}\star_j \nu_{\ell-j}$$
supported on $\langle \Gamma_{\ell-j},\rho_{\ell-j}\rangle<\Gamma'_{\ell-j}$.
Consider the surjective homorphism 
$$\vartheta_{\ell-j}: \langle \Gamma_{\ell-j},\rho_{\ell-j}\rangle \ra 
\mathbb Z\wr \Gamma'_{\ell-j-1}.$$
By Lemma \ref{lem-push1}, this homomorphism can be extended to a surjective 
homomorphism 
$$\vartheta_{\ell-j,j}: W_j(\mathbb Z,\langle 
\Gamma_{\ell-j},\rho_{\ell-j}\rangle) \ra W_j(\mathbb Z, \mathbb Z\wr\Gamma'_{\ell-j-1})=W_{j+1}(\mathbb  Z,\Gamma'_{\ell-j-1}).$$
Further, by Lemmas \ref{lem-push2} and \ref{lem-varthetaconv}, we have 
$$\vartheta_{\ell-j,j}(\widetilde{q}_{j})= q_{j+1}$$
since $q_{j+1}$ is the iterated switch-walk-switch measure on $W_{j+1}(\mathbb Z,\Gamma'_{\ell-j-1})$ based on $\eta$ and $\pi^{\ell-j}_{\ell-j-1}(\phi_{\ell-j})$.
This yields
$$ \widetilde{q}_j^{\star_j 2n}(e_
{\star_j}) \le q_{j+1}^{\star_{j+1} 2n}(e_
{\star_{j+1}}) .
$$
This, together with (\ref{crux1}), proves the desired relation between
$\Phi_{W_j(\mathbb Z, \Gamma'_{\ell-j})}$ and $ 
\Phi_{W_{j+1}(\mathbb Z,\Gamma'_{\ell-j-1})}$.
\end{proof}

\subsection{Free solvable groups}
In this section, we conclude the proof of Theorem \ref{th-freesol} by 
proving that, for $d\ge 3$,
$$\Phi_{\mathbf S_{d,r}}(n)\simeq  
\exp \left(- n \left(\frac{ \log_{[d-1]} n}
{\log _{[d-2]} n }\right)^{2/r}\right).$$
The lower bound follows from Corollary \ref{cor-itlow}. By Theorem 
\ref{theo-itup}, it suffices to construct an exclusive sequence 
$(\Gamma_i,\boldsymbol \rho_i)_2^d$ adapted to $(\Gamma_i(N))_1^d$, when
$N=[\mathbf F_r,\mathbf F_r]$ with the property that $\Gamma'_1$ 
is isomorphic to $\mathbb Z^r$. The technique developed in Section 
\ref{sec-exexclu} 
is the key to constructing such an exclusive sequence. 

In fact, we are able to deal with a class of groups that is more general 
than the family $\mathbf S_{d,r}$. Observe that
$\mathbf S_{d,r}=\Gamma_d(\gamma_2(\mathbf F_r))$. More generally, define
$$\mathbf S^c_{d,r}= \Gamma_d(\gamma_{c+1}(\mathbf F_r))=
\mathbf F_r/(\gamma_{c+1}(\mathbf F_r))^{(d)}.$$
Note that $\Gamma_1(\gamma_{c+1}(\mathbf F_r))=
\mathbf F_r/\gamma_{c+1}(\mathbf F_r)$
is the free nilpotent group of nilpotent class $c$ on $r$ generators.
The groups $\mathbf S^c_{d,r}$ are examples of (finite rank) 
free polynilpotent groups. These groups are studied in \cite{Grue} 
where it is proved that they are residually finite. In the notation of 
\cite{Grue}, $\mathbf S^c_{d,r}$ is a free polynilpotent group of 
class row $(c,1,\dots,1)$ with $d-1$ ones following $c$.

Let 
$$D(r,c)=\sum_1^c\sum_{k|m}\mu(k)r^{m/k}$$
where $\mu$ is the M\"obius function. 
The integer $D(r,c)$ is the exponent of polynomial volume growth of 
the free nilpotent group 
$\mathbf F_r/\gamma_{c+1}(\mathbf F_r)$. 
See \cite[Theorem 11.2.2]{MHall} and \cite{dlH}. Note that $D(r,1)=r$.

\begin{theo} \label{th-fs+}
Fix $c\ge 1$, $r\ge 2$ and $d\ge 3$. Let $D=D(r,c)$. We have
$$\Phi_{\mathbf S^c_{d,r}}(n)\simeq \exp \left(-n \left(\frac{\log_{[d-1]} n}
{\log_{[d-2]}n)}\right)^{2/D}\right). $$
\end{theo}
\begin{rem} The case $d=2$ is covered by  Theorem \ref{theo-nil}. 
\end{rem}

For the proof of Theorem \ref{th-fs+}, we will use a result concerning  
the subgroup of $\Gamma_\ell(N)$ generated by the images of a fix power 
$\mathbf s_i^m$
of the generators $\mathbf s_i$, $1\le i\le r$. Let $\delta _{m}:\mathbf{F}_{r}\rightarrow \mathbf{F}_{r}$ be the
homomorphism from the free group to itself determined by $\delta _{m}(%
\mathbf{s}_{i})=\mathbf{s}_{i}^{m}$, $1\leq i\leq r$.

\begin{lem}\label{lem-stretch}
Suppose $\delta _{m}$ induces an injective homomorphism $\mathbf{F}%
_{r}/N\rightarrow \mathbf{F}_{r}/N$, and $\pi (\mathbf{s}_{i}^{q})\notin
\delta _{m}(\mathbf{F}_{r}/N)$, $1\leq q\leq m-1$, $1\leq i\leq r$. Then $%
\delta _{m}$ induces an injective homomorphism $\mathbf{F}_{r}/[N,N]\rightarrow \mathbf{F}_{r}/[N,N]$.
\end{lem}
\begin{proof} The proof is based on the representation of the elements of 
$\Gamma_2(N)=\mathbf F_r/[N,N]$ using flows on the labeled Cayley graph of 
$\Gamma_1(N)=\mathbf F_r/N$.

Let $\delta_m$ also denote the induced injective homomorphism on $\Gamma_1(N)$.
Let $\mathfrak{f}$ be a flow function defined on edge set $\mathfrak E$ 
of Cayley graph of $%
\Gamma _{1}(N)$. Let $\mathfrak E_{m}$ be a subset of $\mathfrak E$ given by%
\[
\mathfrak E_{m}=\{(\delta _{m}(x)s_{i}^{j},\delta _{m}(x)s_{i}^{j+1},\mathbf{s}%
_{i}):x\in \Gamma _{1}(N),0\leq j\leq m-1,1\leq i\leq r\}. 
\]%
Let $t_{m}: \mathfrak f\mapsto t_m\mathfrak f$ be the  map on flows defined by 
\[
t _{m}\mathfrak{f}((\delta _{m}(x)s_{i}^{j},\delta _{m}(x)s_{i}^{j+1},%
\mathbf{s}_{i}))=\mathfrak{f}((x,xs_{i},\mathbf{s}_{i})),\;\;0\leq j\leq m-1, 
\]%
and $t_{m}\mathfrak{f}$ is $0$ on edges not in $\mathfrak E_{m}$. 
This map
is well-defined. Indeed, if two pairs $(x,j)$ and $(y,j^{\prime })$ in 
$\Gamma_1(N)\times\{0,\cdots, m-1\}$ correspond to a common edge, 
that is, $$(\delta _{m}(x)s_{i}^{j},\delta
_{m}(x)s_{i}^{j+1},\mathbf{s}_{i})=(\delta _{m}(y)s_{i}^{j^{\prime }},\delta
_{m}(y)s_{i}^{j^{\prime }+1},\mathbf{s}_{i}),$$ then $\delta
_{m}(x)s_{i}^{j}=\delta _{m}(y)s_{i}^{j^{\prime }}$, $\delta
_{m}(y^{-1}x)=s_{i}^{j^{\prime }-j}$. Since $\left\vert j^{\prime
}-j\right\vert \leq m-1$, from the assumption $\pi (\mathbf{s}%
_{i}^{q})\notin \delta _{m}(\mathbf{F}_{r}/N)$, $1\leq q\leq m-1$ it follows
that $j^{\prime }=j$. Then $\delta _{m}(y^{-1}x)=\overline{e}$ and, since $%
\delta _{m}$ is injective, we must have $x=y$.

By definition, $t _{m}$ is additive in the sense that $$t
_{m}(\mathfrak{f}_{1}+\mathfrak{f}_{2}\mathfrak{)=}t _{m}\mathfrak{f}%
_{1}+t _{m}\mathfrak{f}_{2}.$$ 
Also, regarding translations in $\Gamma_1(N)$, we have 
$$%
t _{m} \tau _{y}\mathfrak{f}= \tau _{\delta _{m}(y)}t _{m}%
\mathfrak{f}.$$ 
Therefore the identity $\mathfrak{f}_{\mathbf{uv}}=\mathfrak{f}_{%
\mathbf{u}}+\tau _{\pi (\mathbf{u})}\mathfrak{f}_{\mathbf{v}}$, of Remark \ref{rem-alg-a} yields
\[
t _{m} \mathfrak{f}_{\mathbf{uv}}=t _{m}\mathfrak{f}_{\mathbf{u}%
}+\tau _{\delta _{m}(\pi (\mathbf{u}))}t _{m}\mathfrak{f}_{\mathbf{v}}.
\]%
By assumption  $\pi (\delta _{m}(\mathbf{u}%
))=\delta _{m}(\pi (\mathbf{u}))$, therefore%
\[
t _{m}\mathfrak{f}_{\mathbf{uv}}=t _{m}\mathfrak{f}_{\mathbf{u}%
}+\tau _{\pi (\delta _{m}(\mathbf{u}))}t _{m}\mathfrak{f}_{\mathbf{v}}.
\]
This identity allows us to check that the definition of $t _{m}$
acting on flows is consistent with 
$\delta _{m}:\mathbf{F}_{r}\rightarrow \mathbf{F}_{r}$. More precisely, 
for any $\mathbf{g}\in \mathbf{F}_{r}$, we have $$
\mathfrak{f}_{\delta _{m}(\mathbf{g})} = t_{m}\mathfrak{f}_{%
\mathbf{g}}.$$
To see this, first note that this formula holds true on the generators
and their inverses and proceed  by induction on the word length of 
$\mathbf{g}\in \mathbf{F}_{r}$. 

Given $g\in \Gamma_2(N)$, pick a representative $\mathbf g\in \mathbf F_r$
so that $g$ corresponds to the flow $\mathfrak f_\mathbf g$ on $\Gamma_1(N)$.
Define $\widetilde{\delta}_m(g)$ to be the element of 
$\Gamma_2(N)$ that corresponds to the flow $t_m\mathfrak f_\mathbf g=
\mathfrak f_{\delta_m(\mathbf g)}.$ This map is well defined
and satisfies 
$$\widetilde{\delta}_m \circ \pi_2= \pi_2\circ \delta_m.$$
This implies that $\widetilde{\delta}_m: \Gamma_2(N)\ra \Gamma_2(N)$
is an injective homomorphism. Abusing notation, we will drop the 
$\,\widetilde{}\,$ and use the same name, $\delta_m$, for the injective 
homomorphisms $\Gamma_1(N)\ra \Gamma_1(N)$ and $\Gamma_2(N)\ra \Gamma_2(N)$ 
induced by $\delta_m:\mathbf F_r\ra \mathbf F_r$.
\end{proof}

\begin{proof}[Proof of Theorem \ref{th-fs+}]
The lower bound follows from Corollary \ref{cor-itlow}. By Theorem 
\ref{theo-itup}, in order to prove the upper bound, 
it suffices to construct an exclusive sequence 
$(\Gamma_\ell,\boldsymbol \rho_\ell)_2^d$ adapted to $(\Gamma_\ell(N))_1^d$,
$N=\gamma_{c+1}(F_r)$, and  with the property that $\Gamma'_1$ 
is isomorphic to $\Gamma_1(N)=\mathbf F_r/\gamma_{c+1}(\mathbf F_r)$, 
the free nilpotent group of class $c$ on $r$ generators. 

The work of Gruenberg, \cite[Theorem 7.1]{Grue} implies that
$\Gamma_\ell(N)$ is residually finite. Hence the
technique developed in Section \ref{sec-exexclu} apply easily to this 
situation. We are going to use repeatedly Proposition \ref{pro-resf}.

To start, for each $\ell$, we construct an exclusive pair 
$(H_\ell,\boldsymbol \sigma_\ell)$ in $\Gamma_\ell(N)$. Namely,
let $\boldsymbol \sigma_\ell$ be an element in 
$N^{(\ell-1)}\setminus N^{(\ell)}$
in reduced form in $\mathbf F_r$ and such that it projects to a  
non-self-intersecting loop in $\Gamma_{\ell-1}(N)$. Let 
$\mathbf s_{i_1}^k\mathbf s_{i_2}^{\epsilon}$, $\epsilon=\pm 1$, $k\neq 0$, 
be begining of $\boldsymbol \sigma_\ell$. 
Without loss of generality, 
we  assume that $i_1=1$, $i_2=2$ and $\epsilon=1$. Let also $s_i$ and 
$\bar{s}_i$ be the projections of $\mathbf s_i$ onto $\Gamma_\ell(N)$ 
and $\Gamma_{\ell-1}(N)$, respectively.
Let
$(\bar{s}_1^k,\bar{s}_1^k\bar{s}_2,\mathbf s_2)$
be the corresponding edge in $\Gamma_{\ell-1}(N)$. 
Since $\boldsymbol \sigma_\ell$ projects to a simple loop in 
$\Gamma_{\ell-1}(N)$, we must have
$$\mathfrak f_{\boldsymbol \sigma_\ell}((\bar{s}_1^k,\bar{s}_1^k\bar{s}_2,\mathbf s_2)
)\neq 0.$$
Since $\Gamma_{\ell-1}(N)$
is residually finite, there exists a finite index normal subgroup $K_
{\boldsymbol \sigma_\ell} \vartriangleleft 
\Gamma _{\ell-1}(N)$ as in Lemma 
\ref{lem-resf}.  Pick an integer $m_\ell$ such that
$$[\Gamma
_{\ell-1}(N):K_{\mathbf{\sigma }_{\ell}}]\mid m_{\ell}\;\; \mbox{ and }\;\;
\left\vert k\right\vert <m_{\ell}$$ 
and set
$$H _{\ell}=\left\langle s_{i}^{m_{\ell}},1\leq i\leq
r\right\rangle <\Gamma _{\ell}(N).$$
Thinking of $\Gamma_{\ell}(N)$ and $\Gamma_{\ell-1}(N)$ as
$\Gamma_2(N^{(\ell-1)})$ and $\Gamma_1(N^{(\ell-1)})$, respectively, 
Proposition \ref{pro-resf} shows that  $(H _{\ell},
\mathbf{\sigma }_{\ell})$ is an
exclusive pair in $\Gamma _{\ell}(N)$.  

Next, by Lemma \ref{lem-stretch}, for each integer $m$ and each $\ell$, 
the injective homomorphism $\delta_m:\mathbf F_r\ra \mathbf F_r$
induces on $\Gamma_\ell(N)$ an injective homomorphism still denoted by 
$\delta_m: \Gamma_\ell(N)\ra \Gamma_\ell(N)$. For each $1\le \ell\le d$, set
$$M_d=1,\;\;M_\ell= m_{\ell+1} \cdots m_d,$$
and, for $2\le \ell\le d$,
$$ \Gamma_\ell= \delta_{M_{\ell}} (H_\ell) < \Gamma_\ell(N),\;\;
\boldsymbol \rho_\ell=\delta_{M_\ell} (\boldsymbol \sigma_\ell).$$
By construction, $((\Gamma_\ell,\rho_\ell))_2^d$ is an exclusive sequence in 
$(\Gamma_\ell(N))_1^d$ and 
$$\Gamma_1'=\pi^2_1(\Gamma_2)=\langle \bar{s_1}^{M_1}, \dots, \bar{s}^{M_1}_r \rangle < \Gamma_1(N)$$
is isomorphic to $\Gamma_1(N)$ because $\Gamma_1(N)$ is the free nilpotent group on 
$\bar{s}_1,\dots,\bar{s}_r$ of nilpotent class $c$.
\end{proof}


\begin{thebibliography}{10}

\bibitem{Alexpol}
G.~Alexopoulos, \emph{A lower estimate for central probabilities on polycyclic
  groups}, Canad. J. Math. \textbf{44} (1992), no.~5, 897--910. \MR{1186471
  (93j:60007)}

\bibitem{Bendikov}
A.~Bendikov and L.~Saloff-Coste, \emph{Random walks driven by low moment
  measures}, Ann. Probab.

\bibitem{CGP}
T.~Coulhon, A.~Grigor'yan, and C.~Pittet, \emph{A geometric approach to
  on-diagonal heat kernel lower bounds on groups}, Ann. Inst. Fourier
  (Grenoble) \textbf{51} (2001), no.~6, 1763--1827. \MR{1871289 (2002m:20067)}

\bibitem{dlH}
Pierre de~la Harpe, \emph{Topics in geometric group theory}, Chicago Lectures
  in Mathematics, University of Chicago Press, Chicago, IL, 2000. \MR{1786869
  (2001i:20081)}

\bibitem{DLS}
Carl Droms, Jacques Lewin, and Herman Servatius, \emph{The length of elements
  in free solvable groups}, Proc. Amer. Math. Soc. \textbf{119} (1993), no.~1,
  27--33. \MR{1160298 (93k:20051)}

\bibitem{Erschler2006}
A.~Erschler, \emph{Isoperimetry for wreath products of markov chains and
  multiplicity of selfintersections of random walks.}, Probab. Theory Related
  Fields \textbf{136, no.4} (2006), 560–586.

\bibitem{Erschler2004}
Anna Erschler, \emph{Liouville property for groups and manifolds}, Invent.
  Math. \textbf{155} (2004), no.~1, 55--80. \MR{2025301 (2005i:43023)}

\bibitem{Grue}
K.~W. Gruenberg, \emph{Residual properties of infinite soluble groups}, Proc.
  London Math. Soc. (3) \textbf{7} (1957), 29--62. \MR{0087652 (19,386a)}

\bibitem{MHall}
Marshall Hall, Jr., \emph{The theory of groups}, Chelsea Publishing Co., New
  York, 1976, Reprinting of the 1968 edition. \MR{0414669 (54 \#2765)}

\bibitem{Hebish1993}
W.~Hebish and L.~Saloff-Coste, \emph{Gaussian estimates for markov chains and
  random walks on groups.}, Ann. Probab. \textbf{21} (1993), 673--709.

\bibitem{Magnus}
Wilhelm Magnus, \emph{On a theorem of {M}arshall {H}all}, Ann. of Math. (2)
  \textbf{40} (1939), 764--768. \MR{0000262 (1,44b)}

\bibitem{Myasnikov2010}
A.~Myasnikov, V.~Roman'kov, A.~Ushakov, and A.~Vershik, \emph{The word and
  geodesic problems in free solvable groups}, Trans. Amer. Math. Soc.
  \textbf{362} (2010), no.~9, 4655--4682. \MR{2645045 (2011k:20059)}

\bibitem{Pittet2002}
C.~Pittet and L.~Saloff-Coste, \emph{On random walks on wreath products.}, Ann.
  Probab. \textbf{30, no.2} (2002), 948--977.

\bibitem{Pittet2000}
Ch. Pittet and L.~Saloff-Coste, \emph{On the stability of the behavior of
  random walks on groups}, J. Geom. Anal. \textbf{10} (2000), no.~4, 713--737.
  \MR{MR1817783 (2002m:60012)}

\bibitem{Rob}
Derek J.~S. Robinson, \emph{A course in the theory of groups}, Graduate Texts
  in Mathematics, vol.~80, Springer-Verlag, New York, 1993. \MR{1261639
  (94m:20001)}

\bibitem{Sale}
A.~W. {Sale}, \emph{{On the Magnus Embedding and the Conjugacy Length Function
  of Wreath Products and Free Solvable Groups}}, ArXiv e-prints (2012).

\bibitem{SCZ-dv1}
L.~Saloff-Coste and T.~Zheng, \emph{Large deviations for stable like random
  walks on $\mathbb z^d$ with applications to random walks on wreath products},
  2013.

\bibitem{SCnotices}
Laurent Saloff-Coste, \emph{Probability on groups: random walks and invariant
  diffusions}, Notices Amer. Math. Soc. \textbf{48} (2001), no.~9, 968--977.
  \MR{1854532 (2003g:60011)}

\bibitem{Sims}
Charles~C. Sims, \emph{Computation with finitely presented groups},
  Encyclopedia of Mathematics and its Applications, vol.~48, Cambridge
  University Press, Cambridge, 1994. \MR{1267733 (95f:20053)}

\bibitem{VarSendai}
N.~Th. Varopoulos, \emph{Groups of superpolynomial growth}, Harmonic analysis
  ({S}endai, 1990), ICM-90 Satell. Conf. Proc., Springer, Tokyo, 1991,
  pp.~194--200. \MR{1261441 (94m:58217)}

\bibitem{VSCC}
N.~Th. Varopoulos, L.~Saloff-Coste, and T.~Coulhon, \emph{Analysis and geometry
  on groups}, Cambridge Tracts in Mathematics, vol. 100, Cambridge University
  Press, Cambridge, 1992. \MR{1218884 (95f:43008)}

\bibitem{Varnil}
Nicholas~Th. Varopoulos, \emph{Th\'eorie du potentiel sur les groupes
  nilpotents}, C. R. Acad. Sci. Paris S\'er. I Math. \textbf{301} (1985),
  no.~5, 143--144. \MR{801947 (86i:22017)}

\bibitem{Vershik2000}
A.~M. Vershik, \emph{Geometry and dynamics on the free solvable groups}, 2000.

\end{thebibliography}

\providecommand{\bysame}{\leavevmode\hbox to3em{\hrulefill}\thinspace}
\providecommand{\MR}{\relax\ifhmode\unskip\space\fi MR }
\providecommand{\MRhref}[2]{%
  \href{http://www.ams.org/mathscinet-getitem?mr=#1}{#2}
}
\providecommand{\href}[2]{#2}

\end{document}